\documentclass[a4paper,10pt]{article}
\usepackage[centertags]{amsmath}
\usepackage{amsfonts}
\usepackage{amssymb}
\usepackage{amsthm}
\usepackage{epsfig}
\usepackage{setspace}
\usepackage{ae}
\usepackage{eucal}
\usepackage[usenames]{color}%
\usepackage{graphicx}

\theoremstyle{plain}
\newtheorem{thm}{Theorem}[section]
\newtheorem{lem}[thm]{Lemma}
\newtheorem{cor}[thm]{Corollary}
\newtheorem{prop}[thm]{Proposition}
\theoremstyle{definition}
\newtheorem{defi}[thm]{Definition}

\theoremstyle{remark}
\newtheorem{exmp}[thm]{Example}
\newtheorem{rem}[thm]{Remark}

\renewcommand{\div}{\operatorname{div}}

\begin{document}
\title{Construction of unshielded singular solutions of the harmonic field equations}
\author{J\"org Kampen }

\maketitle

\begin{abstract} 
Singular solutions of the harmonic Einstein evolution equation are constructed which are related to spatially global and time-local solutions for a certain class of quasilinear hyperbolic systems of second order.  The constructed singularities of curvature invariants occur generically and are accessible by g.a.p. curves of finite length. The singularities are not strongly censored, and for strongly asymptotically predictable space-times, they are located in the causal past of the future null infinity, and are, hence, not shielded by a black hole. Related ideas may be applied to other hyperbolic equations and, especially, to the Euler equation, but there are special features in case of the Einstein field equation which are considered here in detail. First the data on the Cauchy surface have to be chosen such that the harmonic field equation are well defined in a vicinity of the Cauchy surface in the sense that a uniform Lorentz condition holds. Second we choose data such that curvature invariants blow up at one point in the domain of the metric functions, where the data  have H\"{o}lder continuous first order derivatives everywhere and are smooth in the complement of that one point. Estimates related to convoluted data with Lipschitz continuous first order spatial derivatives are extended to this class of weaker data. Moreover the singular solutions are stable in the sense that  singular solutions are seperated by a ball in some $L^p$ space from any given solution of the Einstein field equation with bounded curvature invariant.
\end{abstract}


2010 Mathematics Subject Classification 35Q76. 
\section{Harmonic Einstein equations, unshielded singularities and cosmic censorship}
In this section we consider the mathematical and physical background of unshielded or 'naked' singularities, which are, according to one attempt of definition, singularities located in the causal past of the future null infinity. The concept of a 'singularity' in classical gravity is elusive as the extensions of different proposes for this concept seem to be either too extensive or too narrow for different reasonable purposes. Even the reasonable concept of shielded singularities just mentioned is a bit narrow in the sense that it is usually defined relative to strongly asymptotically predictable space-times.  Due to this situation, it seems to be easier to prove existence results of singularities than to exclude a certain type of singularities in a broad variety of senses. In order to prove a convincing existence result of unshielded singularities it is sufficient to choose a rather strong concept of a singularity such as the blow up of a curvature invariant along a curve of finite generalised affine parameter length.  Our considerations here are motivated by a certain structure of the Einstein field equation, but similar constructions can be done for a certain class of quasilinear hyperbolic equations of second order as well.
The field equations determine the coefficients $g_{\mu\nu},~0\leq \mu,\nu\leq n$ of the line element
\begin{equation}
ds^2=\sum_{\mu,\nu=0}^n g_{\mu\nu}dx_{\mu}dx_{\nu}
\end{equation}
of a Lorentz manifold $M$, where the zero component refers to time by convention. The usual assumption is that there are three spatial dimensions, i.e., $n=3$, but since the Kalutza-Klein paper appeared there have always been hypotheses around with $n>3$, where the Lorentz metric can be generalized to arbitrary dimension straightforwardly, and no dimension-specific Lorentz-group structure is needed here. It seems reasonable to be not specific with respect to dimension and just assume $n\geq 3$. 

Depending on the nature of singularities considered they are in general located on the boundary $\partial M$ of a manifold $M$ with respect to some topology which has to be defined according to the purposes of the investigation. Such boundaries can be very bizarr and may have counterintuitive properties. Since our intention in this paper is to construct singularities related to curvature blow-ups we may use a rather strong topology. Note that depending on the topology we can include or must exclude (parts of) the boundary from the manifold itself. Especially, if we want basic invariants such as dimension to be well-defined. We better work with $C^p$-manifolds for $p\geq 1$, maybe with exceptions for specific very restricted sets. For, otherwise, we may run into problems concerning the invariance of domain and so on.  We shall consider a rather strong topology for $M\setminus \partial M$ imposed component functions $g_{ij}$ of the metric $g$, which is a covariant $2$-form tensor with Lorentzian signature. The components $g_{ij},~0\leq i,j\leq n$ (where the the component $0$ refers to time) are given in Euclidean coordinates and are in $C^{1,\delta}\left((-\epsilon,\epsilon)\times {\mathbb R}^n,{\mathbb R} \right)$, where the latter space denotes the space of differentiable functions with H\"{o}lder continuous first order derivatives of exponent $\delta\in (0,1)$. 
Here, the local time interval $(-\epsilon,\epsilon)$ for some small $\epsilon >0$ indicates that we shall consider a time-local solution in a neighborhood of a Cauchy surface with a definite Lorentzian signature. The harmonic field equations involve second order derivatives of the metric components such that there is no classical solution of these hyperbolic equations in the function space $C^{1,\delta}\setminus C^2$ for these metric components. In our construction the metric solution has  only one point of space-time in the latter space. This point will be on the boundary of a classical solution. More precisley, the solution of the harmonic field equation assume data $g_{0ij}\in C^{1,\delta}\left({\mathbb R}^n,{\mathbb R} \right)$  for  $1\leq i,j\leq n$, which are smooth in the complement of the origin and in $H^{s+1}\equiv H^{s+1}\left({\mathbb R}^n\right) $ for $s>\frac{n}{2}$. Here the subscript $0$ indicates that time is fixed at $t=t_0$ such that the initial data functions $g_{0ij}:{\mathbb R}^n\rightarrow {\mathbb R}$ are restrictions of the functions $g_{ij}:{\mathbb R}^{n+1}\rightarrow {\mathbb R}$. Related assumptions are made for the first order time derivatives of the initial data (indicated by an upper dot), i.e., $\stackrel{\cdot}{g}_{0ij}\in C^{0,\delta}\left({\mathbb R}^n,{\mathbb R} \right)\cap H^{s} ,~1\leq i,j\leq n$ for some $s>\frac{n}{2}$, and for the first order spatial derivatives of the initial data, i.e., $g_{0ij,k}\in C^{0,\delta}\left({\mathbb R}^n,{\mathbb R} \right)\cap H^s ,~1\leq i,j\leq n$ for some $s>\frac{n}{2}$.
  Note that the curvature invariants involve second order derivatives of the metric tensor, and can, hence, may blow up for some metric components $g_{ij}$ with $g_{ij}\in C^{1,\delta}$. In order to prove time-local existence it is usually assumed there is a Cauchy surface $\Sigma$ and a local time neighborhood of $\Sigma$ such that the metric components have a uniform Lorentzian signature in this neighborhood (cf. \cite{HKM}). 
In this neighboorhood of invariant signature  the metric tensor components $g_{ij}$ satisfy a harmonic field equation on $(-\epsilon,\epsilon)\times {\mathbb R}^n$ with respect to harmonic coordinates $(t,x)$ for $t\in (-\epsilon,\epsilon)$ and for some small $\epsilon >0$. 
   
Note that the loss of well-posedness of the field equations (beyond local-time well-posedness) may be due to singularities or to the loss of a given Lorentzian signature as time passes by. In the following we sometimes use Einstein summation, and use the more classical notation with explicit symbols of sums if we want to emphasize some structure of equations. Notation of ordinary partial derivatives with respect to the variable $x^i$ is either denoted by a subscript $,i$ or by $\frac{\partial}{\partial x^i}$. It is well-known that the field equations can be subsumed under a certain class of quasilinear hyperbolic systems of second order -which were seemingly first studied systematically by Hilbert and Courant. This subsumption is used in \cite{HKM}, but the result obtained on the abstract level is not strong enough for our purposes. For this reason we stick with the special field equations, where we can use special features. In the physical context, as long as considerations of higher dimension seem to be of a speculative type, it seems appropriate to consider the classical field equations in classical space with spatial dimension three and then remark that the result can be generalized (if needed). It is in space-time dimension $3+1$, where calculations based on the field equations produced predictions which were confirmed by experiment. So we think of $n=4$ in general, but keep the treatment general as this costs us nothing.  Recall that the signature of the metric $g_{ij}$ is the number of positive eigenvalues of the matrix $\left( g_{ij}\right)$, i.e., the spatial dimension $n$ in our case, where the index zero is reserved for the time dimension.
In the following representation of the Einstein field equation in (\ref{harm}) Greek indices run from $0$ to $n$ and latin indices run from $1$ to $n$ (cf. also similar notation in \cite{HKM}). For a Lorentz metric on ${\mathbb R}\times {\mathbb R}^n$ we may consider the field equations as a first order quasi-linear hyperbolic system with harmonic coordinates for $g_{\mu\nu},g_{\mu\nu,k},h_{\mu\nu}=\frac{\partial g_{\mu\nu}}{\partial t}$ of the form
\begin{equation}\label{harm}
\left\lbrace \begin{array}{ll}
\frac{\partial g_{\mu\nu}}{\partial t}=h_{\mu\nu}\\
\\
\frac{\partial g_{\mu\nu,k}}{\partial t}=\frac{\partial h_{\mu\nu}}{\partial x^k}\\
\\
\frac{\partial h_{\mu\nu}}{\partial t}=-g_{00}\left(2g^{0k}\frac{\partial h_{\mu\nu}}{\partial x^k}+g^{km}\frac{\partial g_{\mu\nu,k}}{\partial x^m} -2H_{\mu\nu}\right) , 
\end{array}\right.
\end{equation}
with data $g_{\mu\nu}(t_0,.)$ and $h_{\mu\nu}(t_0,.)$ at some time $t_0$, and where
\begin{equation}
\begin{array}{ll}
H_{\mu\nu}\mbox{ is given in (\ref{hmunu}) below.}
\end{array}
\end{equation}
In this context we use the convention
\begin{equation}\label{gamma}
\Gamma^i=g^{\alpha\beta}\Gamma^{i}_{\alpha\beta},
\end{equation}
where we recall that the Christoffel symbols are defined to be
\begin{equation}
\Gamma^{\mu}_{\alpha\beta}=\frac{1}{2}g^{\mu\rho}\left(g_{\rho\alpha,\beta}+g_{\rho\beta,\alpha}-g_{\alpha\beta,\rho} \right).
\end{equation}
As we have have space-time dimension $n+1$ this is a system for
\begin{equation}
\frac{(n+1)(n+2)}{2}(1+n+1)=\frac{(n+1)}{2}(n+2)^2~\mbox{unknowns}~g_{\mu\nu},g_{\mu\nu,k},h_{\mu\nu},
\end{equation}
(or $50$ unknowns in case of space-time of dimension $3+1$).
This system is another way of writing the vacuum field equations
\begin{equation}
R^h_{\mu\nu}=0
\end{equation}
with additional variables of course, where the upper script $h$ indicates that the Ricci tensor $R_{\mu\nu}$ is written in harmonic coordinates. The coordinates are called harmonic because, usually, the  Einstein equations (without energy-momentum source)  are written in coordinates where they take the form
\begin{equation}\label{einsthilb}
G_{\mu\nu}=R_{\mu\nu}-\frac{1}{2}g_{\mu\nu}R=0,
\end{equation}
which contains an additional 'potential' term $\frac{1}{2}g_{\mu\nu}R$. More formally coordinates are called harmonic if $\Gamma^{\mu}(x)=0$, which implies that the coordinate functions themselves are harmonic with respect to the d' Alembert operator. We should note that the vaccum field equations can be written in the form $R_{\mu\nu}=0$ of course, were it was one of the main difficulties to find the right form in the presence of the energy momentum tensor.  In this context recall that the Ricci tensor is given by
\begin{equation}\label{curv}
R_{\mu\nu}=\frac{\partial \Gamma^{\alpha}_{\mu\nu}}{\partial x^{\alpha}}-\frac{\partial \Gamma^{\alpha}_{\alpha\mu}}{\partial x^{\nu}}-\Gamma^{\alpha}_{\mu\nu}\Gamma^{\beta}_{\alpha\beta}-\Gamma^{\alpha}_{\mu\beta}\Gamma^{\beta}_{\nu\alpha},
\end{equation}
and that the scalar curvature is given by
\begin{equation}
R=g^{\mu\nu}R_{\mu\nu}.
\end{equation}
Note that $R$ is a scalar function where the evaluation of a scalar at a point $p\in M$ is denoted by $R_p$.
Historically, it was a major step to find this additional term (saving covariance), and $G_{\mu\nu}$ is called the Einstein tensor. More precisely and for example, in the presence of matter a conservation law should hold such that for (possibly variable) density $\rho_0$ and velocity $v^{\mu},~0\leq \mu\leq n$ the stress-energy-momentum tensor
\begin{equation}
T^{\mu\nu}=\rho_0 v^{\mu}v^{\nu},~0\leq\mu,\nu\leq n,
\end{equation}
satisfies
\begin{equation}
T^{\mu\nu}_{\hspace{0.3cm};\nu}=0,
\end{equation}
and this requirement leads to (\ref{einsthilb}) as we have
\begin{equation}
G^{\mu\nu}_{\hspace{0.3cm};\nu}=0.
\end{equation}
Maybe the tensor $G^{\mu\nu}$ should be called the Einstein-Hilbert tensor, because it is quite possible that Hilbert was the first in November 1915 who wrote the equation  $R_{\mu\nu}-\frac{1}{2}g_{\mu\nu}R=T_{\mu\nu}$ on a blackboard in G\"{o}ttingen in a derivation via variational calculus (with a clear insight that the inhomogeneous term $-\frac{1}{2}g_{\mu\nu}R$ is needed to keep covariance in the presence of matter) \footnote{Note that four pages of Hilbert's corresponding publication are missing in the archive of the academy in Berlin while the variational principle is given in the correct form, and it seems very unlikely that Hilbert did or could not not derive the equations in (\ref{einsthilb}) from the variational principles - probably it is a paragraph in the missing pages.}.
In is a major step to formulate the field equations in the presence of matter. Note again that in the absence of matter we have $g^{\mu\nu}g_{\mu\nu}=n+1$, and therefore we get
\begin{equation}
g^{\mu\nu}\left( R_{\mu\nu}-\frac{1}{2}g_{\mu\nu}R\right)=R-\frac{n+1}{2}R=0,
\end{equation}
hence $R=0$, and the field equations reduce to the vacuum field equations
\begin{equation}
R_{\mu\nu}=0.
\end{equation}
Riemann could have written down them (or may be he has), but Einstein gave meaning to them. The harmonic coordinates used above can be used also in the presence of matter, of course.
In this article we construct singular solutions for the vacuum equations. We note that our method can be applied to extended equations of the form
\begin{equation}
G_{\mu\nu}+\Lambda g_{\mu\nu}=\kappa T_{\mu\nu},
\end{equation}
where $\Lambda$ is a cosmological constant, and $T_{\mu\nu}$ is the energy momentum tensor. Here $\kappa$ is a coupling constant which may be chosen to be $\kappa =1$ (adapting $T_{\mu\nu}$). Such generalisations depend on conditions on the additional terms, of course. For example a positive cosmological constant $\Lambda$ has a damping effect in the region where the metric tensor satisfies a Lorentz condition (and can be written in harmonic form). However, generalisations of the following results are possible for positive and negative cosmological constant. Up to the matter term $T_{\mu\nu}$ the field equations look locally like ordinary wave equations of course (easily to solve), but globally these simple equations are glued together which makes them nonlinear and difficult to solve. Accordingly, most of the research concerns specific solutions to the field equations, while  research of the general equations is more in the context of hyperbolic systems of second order which are investigated by more general methods. Another approach is to study general properties of solutions, of course (cf. \cite{HE}). Results are then applied to the field equations without using their special structure.    For example, in \cite{HKM} it is observed that the field equations for a Lorentz metric can be subsumed by hyperbolic systems of second order of the form
\begin{equation}\label{hypsec}
a_{00}\frac{\partial^2\psi}{\partial t^2}=\sum_{i,j=1}^na_{ij}\frac{\partial^2\psi}{\partial x^i\partial x^j}+\sum_{i=1}^n\left(a_{0i}+a_{i0} \right)\frac{\partial^2\psi}{\partial t\partial x_i}+b, 
\end{equation}
where $\psi=\left(\psi_1,\cdots,\psi_n \right)^T$ is a $n$-vector-valued function of time $t\in [0,T]$ and spatial variables $x=\left(x^1,\cdots,x^n \right)\in {\mathbb R}^n$, and $a_{ij},~1\leq i,j\leq n$ is a collection of $n\times n$-matrix valued functions of suppressed arguments of $t,x,\psi,\frac{\partial \psi}{\partial t},\nabla\psi$, and $b$ is a $n$-vector-valued function of the same arguments (the latter sentence is a citation of \cite{HKM}, p. 274 essentially).   
Local triviality and global complexity are characteristics of the Einstein field equation as their derivation principles are simple (cf. the equivalence principle) while second order tensors like the Ricci tensor (which contain the global information) can have a rich structure, so rich, that general investigations of the field equations work with further assumptions, for example with the assumption of asymptotically predictable space- times. The difficulty of defining certain concepts such as 'singularity', 'black holes', 'weak cosmic censorship' is related to that richness such that these concepts are defined relative to such classes of space-times (such as the mentioned  class of strongly asymptotically predictable space-times). We recall a related class of concepts which lead  us to a concept of black holes, unshielded (naked) singularities, and a concept of weak censorship. We refer the reader to \cite{HE,N1,N2,P,W} for a more detailed discussion of these notions. First we introduce a class of space-times for which black holes are well-defined.
\begin{defi}
A strongly asymptotically predictable space-time is an asymptotically flat space-time $(M,g)$ such that there exists an open region $U$ in the conformal space-time extension $\left( \tilde{M},\tilde{g}\right)$ such that
\begin{itemize}
 \item[i)] $U\supset M\cap J^-\left( I^+ \right)$,
 \item[ii)] $(U,\tilde{g})$ is globally hyperbolic.
\end{itemize}

\end{defi}
In this context of asymptotically predictable space-times black holes can be defined without reference to elusive concept of a singularity. We have

\begin{defi}
The region $B\subseteq M$ is called a black hole of a strongly asymptotically predictable space-time $(M,g)$, if it is the complement of the causal part $J^-$ of the future null infinity $I^+$, i.e., $B=M\setminus J^-\left(I^+\right)$. 
\end{defi}

\begin{defi}
The boundary $H^+:=M\cap\partial J^-\left(I^+\right)$ is the event horizon of a black hole. 
\end{defi}

\begin{defi}
A singularity of space-time is called naked or unshielded if it is located in the the causal past of null infinity $J^-\left(I^+\right)$.
\end{defi}
The weak cosmic censorship conjecture maintains that there is no naked or unshielded  singularity. This concept is due to Hawking and Penrose, of course. It seems to be a rather involved concept, but it is a certain way of making precise Penrose's early statement of 1969 (citation):
\begin{center}
''does there exist a 'cosmic censor' who forbids the appearance of naked singularities closing each one in an absolute event horizon?''
\end{center}
The strong cosmic censorship hypothesis for a metric Lorentzian manifold $\left( M,g_{ij}\right)$ is often defined by strong global hyperbolicity, i.e., the requirement that there is a Cauchy surface $\Sigma$ such that
\begin{equation}
M=D^+\left(\Sigma\right)\cup D^-\left(\Sigma\right), 
\end{equation}
where $D^+\left(\Sigma\right)$ (resp. $D^-\left(\Sigma\right)$) are arcwise connected components separated by $S$ and represent the causal future and the causal past relative to the Cauchy surface $\Sigma$ defined by causal curves. 
The elusiveness of the concept of singularities then leads to a weak interpretation of the concept of a singularity in terms of geodesically incompleteness. Next we discuss some notions of the vague concept of singularities and different attempts to make it precise. This is important in order to understand the role of the Hawking-Penrose theorem, the conjectures of weak and strong cosmic censorship, and the results and arguments of this paper, which can be read as comments on these theorems and claims.
 
Defining singularities by geodesic incompleteness is a well-motivated approach because it seems that other definitions are far too narrow or far to wide. However, we may use a strong definition of singularity and prove its existence for a generic set of Lorentz metrics satisfying the Einstein evolution equation in order to disprove weak cosmic censorship statements which are based on weaker (wider or more extensive) notions of singularities. First let us recall the relevant notions. Note that the line element $ds^2= g_{ij}dx^idx^j$ defines a metric 
\begin{equation}
g:TM\times TM\rightarrow {\mathbb R},
\end{equation}
which is a bilinear form, and where $TM$ denotes the tangential bundle on $M$. Here we should indicate that in the definition of the harmonic field equations   the metric components  $g_{ij}$ are defined with respect to a Euclidean parmetrisation which is mediated locally by charts. This mediatioon by charts is supressed here for brevity.

\begin{defi}
The generalized affine parameter length of a curve $\gamma :[0,c)\rightarrow M$ with respect to a frame $E_s=(e_i(s))_{0\leq i\leq n},¸0\leq s\leq c$ of a family of basis vectors in ${\mathbb R}^n$ (abbreviated by g.a.p.) is given by
\begin{equation}
l_E(\gamma)=\int_0^c\left(\sum_{i=0}^{n}g\left(\stackrel{\cdot}{\gamma(s)},e_i(s)  \right)  \right) ds
\end{equation}

\end{defi}

\begin{defi}
A a curve $\gamma :[0,c)\rightarrow M$ is incomplete if it has finite g.a.p. with respect to some frame, and it is inextensible if there is no limit in $M$ of $\gamma(s)$ as $s$ approaches $c$. Furthermore a space-time is called incomplete if it contains an incomplete inextensible curve.
\end{defi}
Singularities may be approached by future directed incomplete inextensible curves, i.e. curves with nonnegative Lorentz-metric at all points of the curve, and by past-directed inextensible curves, i.e., curves with nonpositive Lorentz-metric at all points of the curve. Future-directed curves starting at a point $p\in M$ are denoted by $I^+(p)$ and past-directed curves are noted by $I^-(p)$. Null-curves or light-curves starting at a point $p$  are curves with zero value of the Lorentz-metric at all points of the curve considered, and are located at the boundary of $I^+(p)$, where it is a matter of taste to include or exclude the boundary of $I^+(p)$ in the definition of $I^+(p)$ (we included this boundary here). The set of singularities related to a Lorentz manifold $(M,g_{\mu\nu})$ which are endpoints of inextensible curves in $I^+(p)$ starting from some point $p\in M$ are denoted by $M^+$, and the set of singularities which are endpoints of inextensible curves in $I^-(p)$ starting from some point $p\in M$ are denoted by $M^-$. For a curve $\gamma :[0,c)\rightarrow M$ with positive $c\in {\mathbb R}\cup \left\lbrace \infty\right\rbrace $ we define
\begin{equation}
I^+(\gamma):=\cup_{q\in \gamma([0,c))}I^+(q),~I^-(\gamma):=\cup_{q\in \gamma([0,c))}I^-(q).
\end{equation}

\begin{defi}\label{geoinc}
A space-time is geodesically incomplete if it contains a geodesic curve which is incomplete.
\end{defi}
The content of the Hawking Penrose theorem is not our main concern here- we refer to \cite{HE} for the precise discussion of its assumptions. We have
\begin{thm}\label{HP}
(Hawking Penrose Theorem) Assume that a time oriented space-time $(M,g_{ij})$ satisfies the conditions
\begin{itemize}
 \item[i)] $R_{ij}X^iX^j\geq 0$ for any non-spacelike vector $X^i$.
 \item[ii)] The timelike and null generic conditions are satisfied.
 \item[iii)] There are no closed timelike curves.
 \item[iv)] One of the three condition holds
 \subitem iva) There exists a trapped surface.
 \subitem ivb) There exists an achronal set without edges. 
 \subitem ivc) There exists a $p\in M$ such that for each future directed null-	
 
 \hspace{1.2cm} geodesic through $p$ the expansion becomes negative.  
\end{itemize}
Then the manifold $(M,g_{\mu\nu})$ contains at least one incomplete timelike or null geodesic.
\end{thm}

Generic occurrence of singularities for the field equations, where 'generic' is to be understood as 'generic relative to physically reasonable space-times', is one feature of the field equations which may be successfully expressed by theorem \ref{HP}. Another proposed feature of Hawking and Penrose is that singularities are (at least weakly) censored, i.e. shielded by black holes or it is even not possible to reach the singularity from any point in space-time by an incomplete curve of finite g.a.p. length.  
It is in this respect that our result provides contradictory evidence. We  Let us first recall the principles of strong and weak cosmic censorships, where we use the usual terminology of the textbooks.

\begin{defi}\label{scc}
(strong cosmic censorship). A  Lorentz space-time manifols $(M,g)$ is strongly censored if it is locally inextendible.  
\end{defi}

\begin{defi}\label{wcc}
(weak cosmic censorship). A singularity point $p\in M$ of a space-time $(M,g)$ is weakly censored if it is not in the causal past of the future null infinity. Accordingly, a space-time is said to have strongly censored singularities if for all Cauchy surfaces of $M$ all singularities in $M^+$ and $M^-$ are strongly censored.  
\end{defi}

The reason for the characterization of singularities by geodesically incomplete curves is that other characterizations turn out to be too narrow or to extensive, but counterexamples concerning the cosmic censorship hypotheses may be based on a narrow concept which is subsumed by the wider concepts proposedin \cite{HE}. 
\begin{defi}\label{ss}
We say that $p\in M^+$ is a strong scalar curvature singularity if there is a incomplete g.a.p. finite  curve $\gamma:[0,c)\rightarrow M$ such that
\begin{equation}
\lim_{s\uparrow c}\gamma(s)=p,
\end{equation}
and
\begin{equation}
\forall \epsilon >0\forall~C>0~\exists~s\in [c-\epsilon,c):~
{\big |}
R_{\gamma(s)}{\big |}\geq C.
\end{equation}
We also say that the scalar curvature invariant blows up at $p$.
\end{defi}
The constructed strong scalar curvature singularities of the field equations below are stable in the sense that for any classical solution $g_{\mu\nu},~0\leq \mu,\nu\leq n$ there is a ball (or a cylinder) in an $L^p$ space (for some $p\geq 1$) of field equation solutions with strong scalar curvature blow up such that the classical solution $g_{\mu\nu},~ 0\leq \mu,\nu\leq n$ is not in this ball or cylinder.
Note that curvature tensor invariants involve second order spatial derivatives of the metric component functions $g_{\mu\nu}$. For this reason there are metric functions which have H\"{o}lder continuous first order derivatives and where curvature invariants blows up. There are even metric functions $g_{\mu\nu}$ which have H\"{o}lder continuous first order derivative and are smooth in the complement of the origin $(0,0)$. First we give an example
\begin{exmp}\label{mainexp}
We observe that data can have weak singularities at a singular point at the boundary of space-time,  where the field equations can still be solved in the complement of this point.  The weak singularities at the origin are located at the boundary of the equation and are not part of the classical solution constructed. The data can be chosen such that they are part of a weak solution of the field equations.
For example, consider a metric of spatial dimension $n=2$ which depends only on one spatial variable, say $x^1$, and is constant with respect to time such that for the spatial indices $1\leq i,j\leq 2$ and for fixed time we have 
\begin{equation}
\left(g_{ij} \right)_{1\leq i,j\leq 2} =\left( \begin{array}{ll}
 1 + \phi_{\delta}(x^1)f(x^1)\hspace{0.5cm} 0\\
 \\
 0 \hspace{2.8cm}1
\end{array}\right) 
,\end{equation}
where $f$ is a univariate function on the field of real number ${\mathbb R}$ of the form  $z\rightarrow f(z)=z^3\cos\left(\frac{1}{z^{\alpha}}\right)$ for $\alpha\in (0.5,1)$ $z\neq 0$ and $f(0)=0$, and $\phi_{\delta}\in C^{\infty}$ is  a function with support $(\-\delta,\delta)$ and with $\phi_{\delta}(0)=1$ (as known for partitions of unity). Note that
\begin{equation}
f\in H^2~\mbox{and}~f\not\in C^2~\mbox{ for }~\alpha \in (0.5,1).
\end{equation}
Since $n=2$ we have $g_{ij}\in H^2$ even for $\alpha \in (0.5,1.5)$.
 
Evaluating the derivatives of $f$ you observe that the second derivative is discontinuous and even blows up at $z=0$. It follows that second order derivatives of the metric $g_{ij}$ with respect to the spatial variable $x^1$ blow up at the origin. As some first order derivatives of the Christoffel symbols in the definition of the Ricci tensor $R_{ij}$ in (\ref{curv}) contain such non-vanishing second order spatial derivatives of the metric $g_{ij}$ a simple calculation shows that the non-vanishing second order derivative terms do not cancel, and as $g_{ij}$ is positive definite and bounded with a bounded inverse $g^{ij}$ the scalar curvature $R=g^{ij}R_{ij}$ blows up at the origin and is smooth in the complement of the origin. Such phenomena are consistent with constraint equations on the Cauchy surface as we shall observe later.
\end{exmp}
The latter example is rather generic. We have
\begin{prop}\label{mainprop}
Let $C^{1,\delta}\equiv C^{1,\delta}\left({\mathbb R}^n\right)$ the space of differentiable functions with H\"{o}lder continuous first order derivatives of H\"{o}lder exponent $\delta \in (0,1)$ or with Lipschitz continuous first order derivatives. Define 
\begin{equation}
|f(x)|_{1,\delta}:=\sum_{0\leq |\alpha|\leq 1}\sup_{y\in {\mathbb R^{n}}}{\big |}D^{\alpha}_xf(y){\big |}+\sup_{y\neq z,~y,z\in {\mathbb R}^n}{\Bigg |}\frac{|D^{\alpha}_xf(z)-D^{\alpha}_yf(y)}{|z-y|^{\delta}}{\Bigg |},
\end{equation}
where $\delta\in (0,1]$.
 For fixed time $t=t_0\in {\mathbb R}$ we consider the metric functions $g^{t_0}_{ij}\equiv g_{ij}(t_0,.),~1\leq i,j\leq n$ for $0\leq i,j\leq n$. Then in any neighborhood of $g^{t_0}_{ij}\in C^2\cap H^2$ in $\left( C^{1,\delta},|.|_{1,\delta}\right)$ there is a function $\tilde{g}^{t_0}_{ij}\in C^{1,\delta}\cap H^2$ such that typical curvature invariants of $\tilde{g}^{t_0}_{ij}$ (such as the scalar curvature) blow up at the origin and are classically well-defined in the complement of the origin.
\end{prop}

\begin{proof}
The scalar curvature satisfies
\begin{equation}
R=g^{\mu\nu}R_{\mu\nu}=g^{\mu\nu}\left( \frac{\partial \Gamma^{\alpha}_{\mu\nu}}{\partial x^{\alpha}}-\frac{\partial \Gamma^{\alpha}_{\alpha\mu}}{\partial x^{\nu}}-\Gamma^{\alpha}_{\mu\nu}\Gamma^{\beta}_{\alpha\beta}-\Gamma^{\alpha}_{\mu\beta}\Gamma^{\beta}_{\nu\alpha}\right) ,
\end{equation}
where the last three terms in
\begin{equation}
\begin{array}{ll}
\Gamma^{\gamma}_{\alpha\beta,\delta}=\frac{1}{2}g^{\gamma\rho}_{,\delta}\left(g_{\rho\alpha,\beta}+g_{\rho\beta,\alpha}-g_{\alpha\beta,\rho} \right)\\
\\
\hspace{1cm}+\frac{1}{2}g^{\gamma\rho}\left(g_{\rho\alpha,\beta,\delta}+g_{\rho\beta,\alpha,j}-g_{\alpha\beta,\rho,\delta} \right)
\end{array}
\end{equation}
contain second order derivatives of the metric. It is essential to consider the local behaviour around the origin. It is always possible to extend such fields such that they have an appropriate behavior at spatial infinity such that the Cauchy problem is well-defined (cf. next section). Let $C>0$ be some constant. 
First, for $\alpha \in (0.5,1)$ consider the univariate function $f:{\mathbb R}\rightarrow {\mathbb R}$ with 
\begin{equation}\label{dataf}
f(z)=\left\lbrace \begin{array}{ll}
      (C+z^3\cos\left(\frac{1}{z^{\alpha}}\right))\phi_{\delta}(z),~ \mbox{if}~z\neq 0\\
      \\
      C~\mbox{if}~z=0.
     \end{array}\right.
\end{equation}
For $z\in (-\delta,\delta)$ the first derivative is
\begin{equation}
f'(z)=\left\lbrace \begin{array}{ll}
      3z^2\cos\left(\frac{1}{z^\alpha}\right)+\alpha z^{2-\alpha}\sin\left(\frac{1}{z^{\alpha}}\right) ,~ \mbox{if}~z\neq 0,\\
      \\
      0~\mbox{if}~z=0.
     \end{array}\right.
\end{equation}
For $z\in (-\delta,\delta)$ the second derivative is
\begin{equation}
f''(z)=\left\lbrace \begin{array}{ll}
      6z\cos\left(\frac{1}{z^{\alpha}}\right)+3 \alpha z^{1-\alpha}\sin\left(\frac{1}{z^{\alpha}}\right)\\
      \\
      +\alpha (2-\alpha) z^{1-\alpha}\sin\left(\frac{1}{z^{\alpha}}\right)-\alpha^2 z^{1-2\alpha}\cos\left(\frac{1}{z^{\alpha}}\right),~ \mbox{if}~z\neq 0,\\
      \\
      \mbox{undefined}~\mbox{if}~z=0.
     \end{array}\right.
\end{equation}
The only singular term is
\begin{equation}
-\alpha^2 z^{1-2\alpha}\cos\left(\frac{1}{z^{\alpha}}\right)
\end{equation}

Hence $f\not\in C^2$ and for $\alpha\in (0.5,1)$ we have a singularity of order $1-2\alpha\in (-0.5,0)$ with upper $L^2$-integrable upper bound $\frac{C}{|z|^{1-2\alpha}}$.
Now consider a $C^{\infty}$- function $\phi_{\delta,\epsilon}:{\mathbb R}\rightarrow {\mathbb R}$ with support in $(-\epsilon,\epsilon)$ and such that $\phi_{\epsilon,\delta}(z)=1$ for $z\in (-\delta,\delta)$. Such functions are well known from the context of partitions of unity. Then in any neighborhood (with respect to  the normed function space stated) of a metric $g_{ij}(t_0,.)$ (evaluated at some time $t_0$) we find a metric $\tilde{g}_{ij}(t_0,.)$ such that for some $\delta >0$ the function
\begin{equation}
x\rightarrow \tilde{g}_{ij}(t_0,x)
\end{equation}
where
\begin{equation}
\tilde{g}_{11}=g_{11}+\delta \cdot\phi_{\delta,2\delta}(x^1)f(x^1)
\end{equation}
stays in the neighborhood and such that the scalar curvature blows up at the origin and is well-defined in the complement of the origin.  Note that we may choose 
$g_{\mu\nu}(0)=\eta_{\mu\nu}>0$ for all 
$\mu,\nu$ such that the inverse of the metric tensor is well defined in a neighborhood of the Cauchy surface and Lipschitz continuity
 of the metric tensor components $g_{\mu\nu}$ implies Lipschitz continuity of the components of the inverse. Here $\eta_{\mu\nu}$ denotes the Lorentz metric.
\end{proof}
Constructions as in (\ref{mainprop}) can be used in order to define data $g_{\mu\nu}(0,.)$ on a Cauchy surface in a subspace $C^{1,\delta}\cap H^2$. We may choose data which are smooth in the complement of the origin, such that curvature invariants blow up at the origin and are well-defined elswhere. There are two additional constraints in the case of the Einstein field equation. First - as is well known - additional constraint equations have to be imposed to make the Cauchy form of the Einstein field equations well-defined. Second, if we want to prove local time existence by local contraction using viscosity limits of extended systems based on the harmonic field equations, then it is necessary to impose additional constraints on the data such that a) the harmonic field equations are well defined and b) a unique Lorentz signature is well defined in the neighborhood of the Cauchy surface. We shall formalize the items a) and b) in the next section. The construction of a local-time contraction is then based on and extended system where viscosity terms $\nu_0 \Delta g_{\mu,\nu}$, $\nu_0 \Delta g_{\mu,\nu,k}$ and $\nu_0 \Delta h_{\mu,\nu}$ with a positive viscosity $\nu_0>0$ are added for $0\leq \mu,\nu \leq n$ and $1\leq k\leq n$.


The upshot of the following considerations is as follows. In the following section we define an appropriate data class (satisfying a uniform Lorentz signature condition and certain integrability conditions) and state the main theorem concerning the existence of time local solutions of the Einstein field equations with strong singularities in the sense of definition \ref{ss}. Here the constraint equations for the harmonic field equations hold in a weak sense at the origin and in the classical sense elsewhere. This result implies that there are counter examples to various interpretations of the strong or weak cosmic censorship hypotheses. Moreover, as we have already remarked, the singularities are stable since classical solutions can ve separated by a $L^p$ ball of solutions with singular curvature invariants.  Note that the  time-local existence theorem is not subsumed by  \cite{HKM}. As a consequence there is also a counter example of the weak cosmic censorship. In the last section we prove the theorem.

\section{Construction of a class of unshielded singular solutions of the harmonic field equation}

Next we are concerned with the main theorem which asserts the existence  of unshielded (naked) singular solutions of the harmonic field equations. We construct a spatially global and time-local solution $g_{\mu\nu},~ 0\leq\mu,\nu\leq n$ on a domain $(-T,T)\times {\mathbb R}^n$ for data $g_{0\mu\nu} \in C^{1,\delta}\cap H^{\frac{n}{2}+1}$ on a Cauchy surface which are smooth in the complement of the origin and such that the functions $g_{\mu\nu}(t,.),~0\leq \mu,\nu \leq n$ are in the Sobolev space $C^2\cap H^{\frac{n}{2}+1}$ for $t\in (-T,0)\cup (0,T)$. The initial data functions have a fixed Lorentz signature on the Cauchy surface. Recall that a Cauchy surface is spacelike, where we may assume that it is given at $x_0=t=0$. We assume that the Lorentz signature remains constant in a neighborhood of the Cauchy surface, i.e., for some $T>0$ and on a domain $(-T,T)\times {\mathbb R}^n$ we have
\begin{equation}
g_{\mu\nu}=\left(S^T\Lambda S \right)_{\mu\nu}, 
\end{equation}
where
\begin{equation}
\begin{array}{ll}
\Lambda=\mbox{diag}((\lambda_j)_{0\leq j\leq n}),~\mbox{sign}(\lambda_0)=-1,~\mbox{sign}(\lambda_j)=1~\mbox{for}~ 1\leq j\leq n.
\end{array}
\end{equation}
The choice of the sign for the time coordinate is a convention, of course.

We can reexpress this condition in terms of the metric tensor.
Note that the $2$-covariant components of the line element $ds^2= g_{ij}dx^idx^j$ determine a bilinear form 
\begin{equation}
g:TM\times TM\rightarrow {\mathbb R},
\end{equation}
on the tangential bundle $TM$ of space-time $M$, where around each point of space time $(t,x)$ we have  local neighborhood $U\subset M$ and $V\subset {\mathbb R}^n$ and  coordinate transformation $i:V\rightarrow U$ with $i(s,y)=(t,x)$ such that for  certain scale with light velocity $c=1$ we have
\begin{equation}\label{gtx}
g_{(t,x)}\circ i(s,y)(z,z)
=
-z_{i_0}^2+\sum_{i\in \left\lbrace 1,\cdots,n \right\rbrace \setminus \left\lbrace i_0\right\rbrace}z_i^2,
\end{equation}
where $i_0$ denotes the index of time. We may also write $g_{(t,x)\mu\nu}\circ i(s,y)(z,z)=\delta_{\mu\nu}z_{\mu}z_{\nu}$. We have chosen $i_0=0$, but this is just by convention, i.e., not determined by the equation itself. For the simple choice of a Cauchy surface $H$ we may choose a corresponding function $i_{H}$ (instance  of the function $i$ in (\ref{gtx})) and then stipulate that the Lorentz signature is strictly constant in the sense that for all $x\in {\mathbb R}^n$ and $z\neq 0$
\begin{equation}\label{gf}
 x\rightarrow \mbox{sign}(g_{(0,x)}\circ i_{H}(0,x)(z,z))
\end{equation}
is a constant function, where for all $x\in {\mathbb R}^n$ and for all $z\neq 0$ we have $g_{(0,x)}\circ i_{H}(0,x)(z,z)\neq 0$. 


In this context we note that we assume a uniform Lorentz condition on the Cauchy surface such that a solution in the vicinity of the Cauchy surface with conatnt Lorentz signature can be constructed. At time $t_0=0$ we may consider diagonalizations of the symmetric matrix $(g_{\mu\nu})(0,.))$  
\begin{equation}
g_{\mu\nu}=\left(S^T\Lambda S \right)_{\mu\nu}, 
\end{equation}
where
\begin{equation}
\begin{array}{ll}
\Lambda=\mbox{diag}((\lambda_j)_{0\leq j\leq n}),~\mbox{sign}(\lambda_0)=-1,~\mbox{sign}(\lambda_j)=1~\mbox{for}\leq j\leq n.
\end{array}
\end{equation}
For asymptotically flat spaces (cf. below) such a unform Lorentz condition is naturally satisfied at spatial infinity. 
Otherwise we may have that $\lambda_i(0,x)\downarrow 0$ as $|x|\uparrow \infty$. 
Then a uniform lorentz condition, say with with $|\lambda_0(0,x)|<\sum_{i=1}^n|\lambda_i(0,x)|$ still holds if $\inf_{x\in {\mathbb R}^n}\frac{\sum_{i=1}^n\lambda_i(0,x)}{|\lambda_0(0,x)|}\geq c>1$
 for some constant $c>0$. We speak of an uniform Lorentz condition in either case.
Consider the third equation in (\ref{harm}) extended by a viscosity term $\nu_0\Delta$ (where $\Delta$ denotes the Laplacian and $\nu_0$ is a positive real viscosity constant), i.e., the equation
\begin{equation}  
\frac{\partial h^{(\nu_0)}_{\mu\nu}}{\partial t}=\nu\Delta h^{(\nu_0)}_{\mu\nu}-g^{(\nu_0)}_{00}\left(2g^{(\nu_0)0k}\frac{\partial h^{(\nu_0)}_{\mu\nu}}{\partial x^k}+g^{(\nu_0)km}\frac{\partial g^{(\nu_0)}_{\mu\nu,k}}{\partial x^m} -2H^{\nu_0}_{\mu\nu}\right),
\end{equation}
where
\begin{equation}\label{hmunu0}
\begin{array}{ll}
H^{(\nu_0)}_{\mu\nu}\equiv H_{\mu\nu}\left(g^{(\nu_0)}_{\alpha\beta},\frac{\partial g^{(\nu_0)}_{\alpha\beta}}{\partial x^{\gamma}} \right)
\end{array}
\end{equation}
and $H_{\mu\nu}$ is defined in (\ref{hmunu}) below. Here we put the viscosity constant $\nu_0$ into brackets in the supperscript in order to avoid confusion with a running index. A local time solution $g^{(\nu_0)}_{\mu\nu},0\leq \mu,\nu\leq n\in C^{1,2}\left((0,T)\times {\mathbb R}^n\right) $ has the representation
\begin{equation}
\begin{array}{ll}
  h^{(\nu_0)}_{\mu\nu}=h^{(\nu_0)}_{\mu\nu}(0,.)\ast_{sp}G_{\nu_0}
  +-g^{(\nu_0)}_{00}\left(2g^{(\nu_0)0k}\frac{\partial h^{(\nu_0)}_{\mu\nu}}{\partial x^k}\right)\ast G_{\nu_0}\\
  \\
 + \left( g^{(\nu_0)km}\frac{\partial g^{(\nu_0)}_{\mu\nu,k}}{\partial x^m}\right) 
 \ast G_{\nu_0} -\left( 2H^{\nu_0}_{\mu\nu}\right)\ast G_{\nu_0}
  \end{array}
\end{equation}
where $G_{\nu_0}$ is the fundamental solution of a heat equation with viscosity constant $\nu_0$ (cf. below) and $\ast_{sp}$ and $\ast$ denote spatial and space-time convolutions respectively. For solutions which vanish at spatial infinities such representations can be rewritten by partial integration and using the convolution rule such that the convoluted functions are terms built by products of the metric components $g^{(\nu_0)}_{\mu\nu}$ its spatial and time derivatives $g^{(\nu_0)}_{\mu\nu,k}$ and $h^{(\nu_0)}_{\mu\nu}$ and corresponding entries of the inverse matrices $g^{(\nu_0),\mu\nu}$, $g^{(\nu_0),\mu\nu}_{,k}$,   and $h^{(\nu_0),\mu\nu}$, and where convolutions involve the Gaussian $G_{\nu_0}$ and first order spatial derivatives of the Gaussian. 
In order to have nontivual solutions of the constraint equations it is desirable to include asymptotically flat spaces into the set of admissible data. Recall that asymptotically flat means that 
the metric $g=\left(g_{\mu\nu}\right)$ has an asymptotic behavior of Schwarzschild type, i.e.,
\begin{equation}
g_{\mu\nu}\sim \eta_{\mu\nu}+\frac{2m}{r}\delta_{\mu\nu},~\mbox{as $r\uparrow \infty$},
\end{equation}
where $m>0$ is a positive constant.
More precisely, for our purposes we may define$H^s$-asmuptotically flat spaces in the following sense.
\begin{defi}\label{defidata}
For a constant $m>0$ and $S>0$ a space time with metric $g_{\mu\nu}$ is called asymptoctically $H^s$-flat
if  for all time $t$ and all $0\leq \mu,\ne\leq n$ we have
\begin{equation}
\delta g_{\mu\nu}(t,.):= g_{\mu\nu}(t,.)-\eta_{\mu\nu}-\frac{2m}{r}\delta_{\mu\nu}\in H^s
\end{equation}
If the latter constion holds for some time $t_0$ on a Cauchy surface, then we say that the metric data are asymptotically flat on this Cauchy surface. 
\end{defi}
In the following we may assume that a Cauchy surface is given at $t_0=0$.
Recall that for two scalar spatial functions
\begin{equation}
\mbox{ if }f,g\in H^s~\mbox{ for }s>\frac{n}{2}~\mbox{ then }fg\in H^s.
\end{equation}
Hence it is natural to require
\begin{equation}\label{idata1}
\delta g^{(\nu_0)}_{\mu\nu}(0,.)=\delta g^{(\nu_0)}_{\mu\nu}(0,.)\in H^{s+1}\cap C^{s+1}~\mbox{ for some $s>\frac{n}{2}$}
\end{equation}
and 
\begin{equation}\label{idata2}
\delta h^{(\nu_0)}_{\mu\nu}(0,.)=\delta h_{\mu\nu}(0,.)\in H^{s}~\mbox{ for some $s>\frac{n}{2}$}.
\end{equation}
\begin{rem}
The requirements in (\ref{idata}) and (\ref{idata2}) are stronger than needed (cf. \cite{KRS}). However, for our purposes these assumptions are sufficient, and simplify the proof. 
\end{rem}
In addition we require that there is a ball (with respect to supremum norm) of constant Lorentz signature around the initial data, where in case $n=3$ the Lorentz signature is $(-+++)$ and sinilarly for $n>3$. 
\begin{defi}\label{initdatadef}
We call a list of data $g^{(\nu_0)}_{\mu\nu}(0,.),~0\leq \mu,\nu\leq n$ admissible if a uniform Lorentz condition is satisfied on the Cauchy surface (initial data surface), and if in additions the conditions (\ref{idata1}), (\ref{idata2}) are satisfied. We remark that this definition is given with respect to the hyperplane $x_0=t_0=0$, but the form of the statement is chosen such it can be generalised to spacelike Cauchy surfaces straightforwardly. In general  for initial time $t_0$ we define for  $s>0$, and finite constants $c_1,c_2 \geq 0$ the set
\begin{equation}
AD^s_{c_1c_2}:=\lbrace g_{0\mu\nu}|\delta g_{0\mu\nu}(t_0,.):= g_{\mu\nu}(t_0,.)-c_1\eta_{\mu\nu}-c_2\frac{2m}{r}\delta_{\mu\nu}\in H^s\rbrace,
\end{equation}
where  $AD^s_{c_1c_2}$ is called admissible if $s>\frac{n}{2}+1$.
\end{defi}
Admissible data lists are abundant, although they are a bit specific concerning the behaviour at spatial infinity. These requirements allow us to concentrate on the construction of naked singularity, where the requirements at spatial infinity are similar to artificial boundary conditions, which facilitate the prove of the local contraction theorem.  
Note that in the given setting the harmonic field equations can be subsumed by a class of quasilinear hyperbolic equations of second order. There is a local existence theory for such type of equations (cf. \cite{HKM}), which proves the existence of a time-local and unique solution up to a small time horizon $T>0$ for regular data, where the components of the metric 
related assumptions for the first order spatial and time derivatives). However, this local existence result cannot be applied in our situation, because we have only $C^{1,\delta}$ regularity at one point and for some examples  in this class second order derivatives may be 
even not integrable. There are two possible reasons for time locality of existence theorems. One possible reason is that a solution 'develops' singularities after finite time. The other reason is that the metric does not satisfy a Lorentz condition after some time, and then 
cannot be subsumed under the type of quasilinear hyperbolic systems for which the local existence results are proved.  However, if a strict Lorentz condition is satisfied uniformly on a Cauchy surface $C$ at time $t=t_0$, then it holds in a neighborhood for time  $t\in (t_0-\rho,t_0+\rho)$ for some $\rho >0$. 

The uniform Lorentz condition above ensure that this requirement for the existence of a local solutions satisfied. The harmonic form of the Einstein field equations confirm time symmetry as a natural property of the theory. For this reason of time symmetry (which we find in many hyperbolic equations of mathematical physics) we can build in a weak singularity (here a curvature singularity) into the initial data on the Cauchy surface and then show that there exists a time-local solution. This local solution can be extended locally to past time and to future time. An alternative approach is to start with smooth data on the Cauchy surface, and then show that a singularity can develop at the tip of a cone. We considered such a construction for the vorticity form of the incompressible Euler equation elsewhere. Note that the alternative construction considered here may be applied to the Euler equation as well, although the Leray projection term and special form of the equation leads to a  different situation. The vorticity form of the Euler equation together with incompressibility and a different Laplacian kernel (in case $n=2$ and in case $n\geq 3$) lead to specific constraints in the case of dimension $n=2$, such that singularities can be observed for $n\geq 3$ for the Cauchy problem with regular data. In contrast for the Einstein field equation  Cauchy problems with regular data can have singular solutions in any dimension $n \geq 2$.

In the following we use the term 'classical solution of a differential equation' on a certain domain in the usual sense that a solution function satisfies the differential equation pointwise, and  where the (partial) derivatives exist in the classical Weierstrass sense. 

Finally in order to have a well-defined Cauchy problem we need that some constrain equations are satisfied. Four\`{e}s-Bruhat observed that constraint condition $\Gamma^{\mu}=0,~0\leq \mu\leq n$ on the Cauchy surface are automatically transferred to $t>0$ as long as a solution exists. Therefore solutions of the harmonic field equations $R^{h}_{\mu\nu}=0$ are well defined if the initial data constraints $G^0_{\mu}(0,x)=0,~0\leq \mu\leq n$ are satisfied where $G^0_{\mu}$ is defined via the Einstein tensor $G_{\mu\nu}$. Note that these simplified constraint equations hold in case of the vaccum field equations.
\begin{rem}
In the presense of matter the general constraint equations take the form
\begin{equation}
G_{\mu 0}=\kappa T_{0\mu },~0\leq \mu\leq n,~ \kappa=8\pi G.
\end{equation}
In the literature these equations are often expressend in terms of the second fundamental form $K$ of the Cauchy surface, where the relations
\begin{equation}
\begin{array}{ll}
G_{00}=\frac{1}{2}\left(R^{\Sigma}+\left(\mbox{tr}_{\gamma}K \right)^2-|K|^2\right),\\
\\
G_{i0}=K^j_{i;j}-K^j_{j,i} 
\end{array}
\end{equation}
are used. Here, $R^{\Sigma}$ denotes the Riemann tensor on a $n$-dimensional Cauchy surface $\Sigma$, and  the second fundamental form $K$ has the components 
\begin{equation}
K_{ij}:=K(\partial_i,\partial_j)=\frac{1}{2}\left(N^{-1}\left(\partial_i\gamma_{ij}-L_X\gamma_{ij}\right)  \right) 
\end{equation}
with respect to the metric
\begin{equation}
g=-Ndt^2+\gamma_{ij}\left(dx^i+X^idt \right)\left(dx^j+X^jdt  \right), 
\end{equation}
(Einstein summation) and with respect to the Lie derivative
\begin{equation}
L_X\gamma_{ij}=\nabla_iX_j+\nabla_jX_i,~(\nabla_k\mbox{ denotes the covariant derivative})
\end{equation}
 (cf. \cite{CF}, especially the first article in this book for a general treatment of the constraint equations and its solutions). In this paper there are only two issues concerning the constraint equations: a) if we define singular data, i.e., data with one singular point, then they should be defined such that the constraint equations are satisfied on the Cauchy surface in the complemetary domain of the singular point; b) the constraint equations should be satisfied on the whole Cauchy surface $t_0=0$ auch that the the result of Four\`{e}s-Bruhat is satisfied, i.e., such that the constined equations hold on the Cauchy surfaces at $t\in(0,T]$. We shall observe that both requirements are satisfied below.
\end{rem}

Since the construction of naked or unshielded singularities in \cite{Cr1} is unstable as shown in \cite{Cr2}we are interested in stable singularities of solutions. For a given bounded open set $U\subset {\mathbb R}^n$ with respect to the standard topology let $C_b(U)$ denote the space of bounded continuous functions.

\begin{defi}\label{defising}
A singularity at $(t_0,x)$ of a solution $g_{\mu\nu},~ 0\leq \mu,\nu \leq n$ is called generically stable in some $L^p$-sense for some $p\geq 1$ if for any bounded open neighborhood $x\in U\subset {\mathbb R}^n$ and any given bounded continuous  scalar function $f\in C_b,~ r\rightarrow f(r),~ r=\sqrt{\sum_{j=1}^nx_j^2}$  the scalar curvature $R$ of $g_{\mu\nu},~0\leq \mu,\nu\leq n$ at time $t_0$ is seperated from $f$ by an $L^p$-ball   in the sense that for all $c>0$ there exists a $p\geq 1$ such that
\begin{equation}
\int_U|R(x)-f(r)|^pdx\geq c.
\end{equation} 
\end{defi}

Some notation before we state the main theorem:
we write
\begin{equation}
g_{0\mu\nu}(0,.)\in C^{\infty}\left({\mathbb R}^n\setminus \left\lbrace (0,0)\right\rbrace  \right)\mbox{ if $g$ is smooth at all}~y\in {\mathbb R}^n\setminus \left\lbrace (0,0)\right\rbrace.
\end{equation}

We have
\begin{thm}\label{mainthm1}
Consider the case $n=3$ (for generalisations to $n>3$ cf. remark \ref{ngreater3} below). For a list of admissible data functions (cf. the definition in (\ref{defidata}) above)
\begin{equation}
\begin{array}{ll}
g_{\mu\nu}(0,.)=g_{0\mu\nu}:{\mathbb R}^n\rightarrow {\mathbb R},~g_{\mu\nu,t}(0,.)=h_{0\mu\nu}:{\mathbb R}^n\rightarrow {\mathbb R},\\
\end{array}
\end{equation}
where
\begin{equation}
\begin{array}{ll}
g_{0\mu\nu}(0,.)\in C^{\infty}\left({\mathbb R}^n\setminus \left\lbrace (0,0)\right\rbrace  \right)\cap C^{1,\delta}\left({\mathbb R}^n\right) \\
\end{array}
\end{equation}
there is a time-local solution to the Cauchy problem 
\begin{equation}\label{harm2}
\left\lbrace \begin{array}{ll}
\frac{\partial g_{\mu\nu}}{\partial t}=h_{\mu\nu}\\
\\
\frac{\partial g_{\mu\nu,k}}{\partial t}=\frac{\partial h_{\mu\nu}}{\partial x^k}\\
\\
\frac{\partial h_{\mu\nu}}{\partial t}=-g_{00}\left(2g^{0k}\frac{\partial h_{\mu\nu}}{\partial x^k}+g^{km}\frac{\partial g_{\mu\nu,k}}{\partial x^m} -2H_{\mu\nu}\right)\\
\\
 g_{\mu\nu}(0,.)=g_{0\mu\nu},~g_{\mu\nu,t}(0,.)=h_{0\mu\nu},~g_{\mu\nu,k}(0,.)=g_{0\mu\nu,k}.
\end{array}\right.
\end{equation}
on a time horizon $[0,T]$ for some $T>0$ which is classical in the complement of the origin. Here, the harmonic term is given by
\begin{equation}\label{hmunu}
\begin{array}{ll}
H_{\mu\nu}\equiv H_{\mu\nu}\left(g_{\alpha\beta},\frac{\partial g_{\alpha\beta}}{\partial x^{\gamma}} \right)=g^{\alpha\beta}g_{\delta\epsilon}\Gamma^{\delta}_{\mu\beta}\Gamma^{\epsilon}_{\nu\alpha}
\\
\\
+\frac{1}{2}\left(\frac{\partial g_{\mu\nu}}{\partial x^{\alpha}}\Gamma^{\alpha}+g_{\nu\rho}\Gamma^{\rho}_{\alpha\beta}g^{\alpha\eta}g^{\beta\sigma}\frac{\partial g_{\eta\sigma}}{\partial x^{\mu}} +g_{\mu\rho}\Gamma^{\rho}_{\alpha\beta}g^{\alpha\eta}g^{\beta\sigma}\frac{\partial g_{\eta\sigma}}{\partial x^{\nu}}\right),
\end{array}
\end{equation}
where $\Gamma^{\alpha}$ is defined as in (\ref{gamma}).
Moreover, for admissible sets of data $g_{ij0},~1\leq i,j\leq n$ in the functions space $\left( C^{1,\delta}\cap H^s,|.|_{1,\delta}\right)$ the scalar curvature blows up at the origin $(t,x)=0$, and this singularity of space-time is generically stable in some $L^p$-sense as defined in definition \ref{defising} above. The time-local solutions constructed satisfy the usual constraint equations on Cauchy surfaces located at time $t_0>0$. The classical solution breaks down at  the origin (which is part of the boundary of the Lorentzian of the manifold). The constructed solutions for the harmonic field equations $R^h_{\mu\nu}=0$ and the constraint equations $G^0_{\mu}(0,x)=0,~0\leq \mu\leq n$ are weak solutions on the whole domain (spatially in $H^2$), where the origin is included.
\end{thm}
\begin{rem}\label{ngreater3}
For asymptotically flat spaces the initial data include terms of Schwarzschild asymptotics $O\left( \frac{1}{r}\right)$ which are clearly not in $H^s$ for $s>\frac{n}{2}$. For $n=3$ the construction scheme below holds. For space dimension $n>3$ the existence results holds if the Scharzschild term is eliminated.    
\end{rem}
\begin{rem}
A solution $g_{\mu\nu},~0\leq \mu,\nu\leq n$ of the system (\ref{harm2}) is called 'classical' on the domain $\left(0,T\right]\times {\mathbb R}^n$ if $g_{\mu\nu}\in C^{2,2}$, where $C^{2,2}$ is the function space of twice continuously differentiable functions, where the first superscript refers to time and the second superscript refers to the spatial variables. Note that $h_{\mu\nu}=g_{\mu\nu,t}$ is encoded by the metric $g_{\mu\nu},~0\leq \mu,\nu\leq n$ such that we may refer to the metric function components alone as a solution of the system in (\ref{harm2}).  
\end{rem}

\begin{rem}
Although the target is to construct $g_{\mu\nu}\in C^{2,2}$ in the complement of the origin we note that a solution $g_{\mu\nu}\in C^{1,2}$ is a classical solution of the system in (\ref{harm2}) in the domain $(0,T]\times {\mathbb R}^n$, as it contains only spatial derivatives of the metric components $g_{\mu\nu}$ up to second order and spatial derivatives of the first order time derivatives $h_{\mu\nu}$ even up to first order. Hence, if $g_{\mu\nu}\in C^{1,2}\left((0,T]\times {\mathbb R}^n\right)$ is known to be a solution of some system equivalent to (\ref{harm2}), then the third equation in (\ref{harm2}) actually tells us that $h_{\mu\nu,t}$ is continuous $(0,T]\times {\mathbb R}^n$ such that the solution of the original system is actually in $C^{2,2}\left((0,T]\times {\mathbb R}^n\right)$.
\end{rem}

We have to mention that any solution of the Einstein field equation has to satisfy some constraint equations. Since the Einstein evolution is time reversible it is sufficient to have the  the Hamiltonian constraint equation and the momentum constraint equation satisfied at some time $t_0>0$, where the solution is regular (cf. also the remark \ref{constr} below). This can be achieved by smoothing the data such that the constraint equations are satisfied for the smoothed data. The local-time solution for smooth data then satisfies the constraint equation at each time section $t_0$ of the solution interval $[0,T]$. It is straightforward to observe that in the limit of smooth data to $H^2\cap C^{1,\delta}$ data these constraint equations are the still satisfied for $t_0>0$ and hold in $H^2$ sense at $t_0$. Note here that the singular perturbation (singular in the sense of adding a term in $C^{1,\delta}\setminus C^2$ to the initial data) in the proof of Proposition \ref{mainprop} is in a $1$-dimensional subspace.    
As the constructed metric solutions $g_{ij}$ with singular scalar curvature are bounded on the domain of well-posedness, it is clear that there are causal curves of finite g.a.p. length (in the domain where the solutions exists) which reach the point of singular scalar curvature at the origin. We get

\begin{cor}\label{CC}
 Theorem \ref{mainthm1} implies that the weak cosmic censorship in the sense of definition  definition \ref{wcc} is violated. Here, the unshielded singularity of the curvature invariant $R$ is generically $L^p$ stable in the sense of Definition \ref{defising}. Moreover it shows that the Einstein field eqiuation are not globally deterministic in $H^2$. This indicates that the principle of strong cosmic censorship also fails.
\end{cor}

\begin{rem}\label{constr}
We discussed that the Einstein evolution equation imposes constraint equations on the Cauchy surface for the sake of well-posedness. Note that these equations should be satisfied in the viscosity limit of the constructed local solution. We have a  Hamilton constraint equation and a momentum constraint equation which we can tacitly assume to be satisfied if the data are regular, i.e. in $C^2\cap H^2$ (at least). If the constraint equations are satisfied for regular data on a Cauchy surface then they are satisfied on all Cauchy surfaces of the time evolution of the Einstein field equations. We can satisfy the constaint equation for smoothed data  on the Cauchy surfaces (corresponding to time parameter $t_0$). As the smoothed data $t_0$ converge to data with property as in Theorem  \ref{mainthm1} the constrain equations remain satisfied pointwise for time parameter $t>t_0$ and in distributional sense at $t_0$. As the constraint equations are local they also remain satisfied pointwise in the complementary region of the singularity.
\end{rem}

\section{Proof of theorem \ref{mainthm1}}
The proof is based on a local contraction argument for the component functions $g_{\mu\nu}$, $g_{\mu\nu,k}$ and $h_{\mu\nu}$ in an appropriate function space. Note that the Lorentz metric is nondegenerate in the vicinity of the Cauchy surface due to the uniform Lorentz condition and that the entries of the inverse matrix of the metric tensor are regular for the list of admissible data. These admissible data lists are defined such that local the iteration procedures defined below in the vicinity of a Cauchy surface are well-defined.

We start our analysis for admissible  data $g_{0\mu\nu}\in AD^s_{00},~0\leq \mu,\nu\leq n$ for $s>\frac{n}{2}+1$. Later we generalise to the case of general admissible data in $AD^s_{c_1c_2}$ for  $c_1,c_2\geq 0$.

We choose data such that a) the constraint equations are satisfied, and b) the curvature invariant blows up. The singularity can be defined by adding a singular term around zero to an regular admissible data set in a neighborhood around the origin (as in Example \ref{mainexp} and in Proposition \ref{mainprop}). This singularity may be appear for the second order spatial derivatives of some metric component data only, i.e., in case $n=3$ we may choose for example indices $0\leq \mu,\nu\leq 2$ and data $g_{0\mu\nu}\in C^{1,\delta},~ 2\leq \mu,\nu\leq n$ and ensure that the Riemann curvature invariant $R$ blows up at the origin (as in the mentioned examples the factor $\phi_{\delta}$ in the additional term ensures that the singularity of $R$ is located in a neighborhood $B_{\delta}$, i.e., in a ball of radius $\delta >0$ around the origin). 
The other metric components then are regular. One way to observe this is the following. We have a local representation 
\begin{equation}
g_{\mu\nu}=\eta_{\mu\nu}+\delta_{\mu\nu}-\frac{1}{2}\eta_{\mu\nu}\delta,~ \delta=\delta^{\mu}_{\hspace{0.1cm}\mu},~ \delta_{\mu\nu}\ll 1,
\end{equation}
which describes the field in a small domain $B_{\delta}$.
The Hilbert gauge
\begin{equation}
\delta^{\alpha\beta}_{\hspace{0.3cm},\beta}=0
\end{equation}
implies that the perturbations satisfy the equations
\begin{equation}\label{loceq}
\square \delta_{\mu\nu}=-16\pi G T_{\mu\nu}.
\end{equation}
Standard regularity results of the Poisson equation then imply that for $2\leq \mu,\nu\leq3$ the metric components are regular, i.e., $g_{\mu\nu}\in H^2\cap C^2$. This construction shows a) that singularities may be restricted to some submanifolds, and b) that  constraint equations can be satisfied where  the Riemann curvature invariant $R$  blows up. 
In the following we assume $T_{\mu\nu}\equiv 0$ without loss of generality.

Since the initial data have a weak regularity at one point we do not claim uniqueness but construct a local time solution (branch).
We construct local solutions via viscosity limits $\nu_0\downarrow 0$ of local solutions of the extended system ($\nu_0>0$ a small positive 'viscosity' parameter)
\begin{equation}\label{harm3}
\left\lbrace \begin{array}{ll}
\frac{\partial g^{(\nu_0)}_{\mu\nu}}{\partial t}=
\nu_0 \Delta g^{(\nu_0)}_{\mu\nu}+h^{(\nu_0)}_{\mu\nu},\\
\\
\frac{\partial g^{(\nu_0)}_{\mu\nu,k}}{\partial t}=\nu_0 \Delta g^{(\nu_0)}_{\mu\nu,k}+\frac{\partial h^{(\nu_0)}_{\mu\nu}}{\partial x^k},\\
\\ 
\frac{\partial h^{(\nu_0)}_{\mu\nu}}{\partial t}=\nu_0 \Delta h^{(\nu_0)}_{\mu\nu}\\
\\
-g^{(\nu_0)}_{00}\left(2g^{(\nu_0)0k}\frac{\partial h^{(\nu_0)}_{\mu\nu}}{\partial x^k}+g^{(\nu_0)km}\frac{\partial g^{(\nu_0)}_{\mu\nu,k}}{\partial x^m} -2H^{(\nu_0)}_{\mu\nu}\right),\\
\\
 g^{(\nu_0)}_{\mu\nu}(0,.)=g_{0\mu\nu},~g^{(\nu_0)}_{\mu\nu,t}(0,.)=h_{0\mu\nu},~g^{(\nu_0)}_{\mu\nu,k}(0,.)=g_{0\mu\nu,k}.
\end{array}\right.
\end{equation}
For a function  $f:(0,\rho)\times {\mathbb R}^n\rightarrow {\mathbb R}$ on some time horizon $\rho >0$ we abbreviate
\begin{equation}
f\ast G_{\nu_0}=\int_0^.\int_{{\mathbb R}^n}f(s,y)G_{\nu_0}(.-s,.-y)dyds
\end{equation}
for the convolution with the Gaussian $G_{\nu_0}$, and for a function $f_0:{\mathbb R}^n\rightarrow {\mathbb R}$ we write
\begin{equation}
f_0\ast_{sp} G_{\nu_0}=\int_{{\mathbb R}^n}f_0(y)G_{\nu_0}(.,.-y)dy
\end{equation}
for the convolution with the Gaussian restricted to the spatial variables.

We define an iterative scheme $g^{(\nu_0)l}_{\mu\nu},~g^{(\nu_0)l}_{\mu\nu,k},~h^{(\nu_0)l}_{\mu\nu},~0\leq \mu,\nu\leq n,~1\leq k\leq n,~l\geq 0$ with integration index $l\geq 0$.  We may initialize the scheme with
\begin{equation}
g^{(\nu_0)0}_{\mu\nu}=g_{0\mu\nu}\ast_{sp} G_{\nu_0},~g^{(\nu_0)0}_{\mu\nu,k}=g_{0\mu\nu,k}\ast_{sp} G_{\nu_0},~h^{(\nu_0)0}_{\mu\nu}=h_{0\mu\nu}\ast_{sp} G_{\nu_0},
\end{equation}
where $G_{\nu_0}$ is the fundamental solution of
\begin{equation}
G_{\nu,t}-\nu \Delta G_{\nu_0}=0.
\end{equation}
For $l\geq 1$ the functions $g^{(\nu_0)l-1}_{\mu\nu},~g^{(\nu_0)l-1}_{\mu\nu,k},~h^{(\nu)l-1}_{\mu\nu},~0\leq \mu,\nu\leq n,~1\leq k\leq n$ are given, and the functions $g^{(\nu_0)l}_{ij},~g^{(\nu_0)l}_{ij,k},~h^{(\nu)l}_{ij},~0\leq \mu,\nu\leq n,~1\leq k\leq n$ are defined iteratively as solutions of the equation
\begin{equation}\label{harm3l}
\left\lbrace \begin{array}{ll}
\frac{\partial g^{(\nu_0)l}_{\mu\nu}}{\partial t}=\nu_0 \Delta g^{(\nu_0)l}_{\mu\nu}+h^{(\nu_0)l-1}_{\mu\nu},\\
\\
\frac{\partial g^{(\nu_0)l}_{\mu\nu,k}}{\partial t}=\nu_0 \Delta g^{(\nu_0)l}_{\mu\nu,k}+\frac{\partial h^ {(\nu_0)l-1}_{\mu\nu}}{\partial x^k},\\
\\ 
\frac{\partial h^{(\nu_0)l}_{\mu\nu}}{\partial t}=\nu_0 \Delta h^{(\nu_0)l}_{\mu\nu}\\
\\
-g^{(\nu_0)l-1}_{00}\left(2g^{(\nu_0)l-1,0k}\frac{\partial h^{(\nu_0)l-1}_{\mu\nu}}{\partial x^k}+g^{(\nu_0)l-1,km}\frac{\partial g^{(\nu_0)l-1}_{\mu\nu,k}}{\partial x^m} -2H^{(\nu_0)l-1}_{\mu\nu}\right),\\
\\
 g^{(\nu_0)l}_{\mu\nu}(0,.)=g_{0ij},~g^{(\nu_0)l}_{\mu\nu,t}(0,.)=h_{0ij},~g^{(\nu_0)l}_{\mu\nu,k}(0,.)=g_{0\mu\nu,k},
\end{array}\right.
\end{equation}
where
\begin{equation}
\begin{array}{ll}
H^{(\nu_0)l-1}_{\mu\nu}\equiv H^{(\nu_0)l-1}_{\mu\nu}\left(g^{(\nu_0)l-1}_{\alpha\beta},\frac{\partial g^{(\nu_0)l-1}_{\alpha\beta}}{\partial x^{\gamma}} \right)\\
\\
=g^{(\nu_0)l-1,\alpha\beta}g^{(\nu_0)l-1}_{\delta\epsilon}\Gamma^{(\nu_0)l-1,\delta}_{i\beta}
\Gamma^{(\nu_0)l-1,\epsilon}_{j\alpha}
\\
\\
+\frac{1}{2}{\Big(}\frac{\partial g^{(\nu_0)l-1}_{ij}}{\partial x^{\alpha}}\Gamma^{(\nu_0)l-1,\alpha}+g^{(\nu_0)}_{j\rho}\Gamma^{(\nu_0)l-1,\rho}_{\alpha\beta}g^{(\nu_0)l-1,\alpha\eta}g^{(\nu_0)l-1,\beta\sigma}\frac{\partial g^{(\nu_0)l-1}_{\eta\sigma}}{\partial x^i}\\
\\
 +g^{(\nu_0)l-1}_{i\rho}\Gamma^{(\nu_0)l-1,\rho}_{\alpha\beta}g^{(\nu_0)l-1,\alpha\eta}g^{(\nu_0)l-1,\beta\sigma}\frac{\partial g^{(\nu_0)l-1}_{\eta\sigma}}{\partial x^j}{\Big )}.
\end{array}
\end{equation}
Here,
\begin{equation}
\Gamma^{(\nu_0)l-1,\mu}=g^{(\nu_0)l-1,\alpha\beta}\Gamma^{(\nu_0)l-1,\mu}_{\alpha\beta},
\end{equation}
and
\begin{equation}
\Gamma^{(\nu_0)l-1,\mu}_{\alpha\beta}=\frac{1}{2}g^{(\nu_0)l-1,\mu\rho}\left(g^{(\nu_0)l-1}_{\rho\alpha,\beta}+g^{(\nu_0)l-1}_{\rho\beta,\alpha}-g^{(\nu_0)l-1}_{\alpha\beta,\rho} \right).
\end{equation}
We have the representation
\begin{equation}\label{repl}
\begin{array}{ll}
 g^{(\nu_0)l}_{\mu\nu}=g_{0\mu\nu}\ast_{sp} G_{\nu_0}+h^{(\nu_0)l-1}_{\mu\nu} \ast G_{\nu_0},~\\
 \\
 g^{(\nu_0)l}_{\mu\nu,k}=g_{0\mu\nu,k}\ast_{sp} G_{\nu_0}+\frac{\partial h^ {(\nu_0)l-1}_{\mu\nu}}{\partial x^k}\ast G_{\nu_0}\\
 \\
  h^{(\nu_0)l}_{\mu\nu}=h_{0\mu\nu}\ast_{sp} G_{\nu_0}
  -g^{(\nu_0)l-1}_{00}\times\\
  \\
  \times \left(2g^{(\nu_0)l-1,0k}\frac{\partial h^{(\nu_0)l-1}_{\mu\nu}}{\partial x^k}+g^{(\nu_0)l-1,km}\frac{\partial g^{(\nu_0)l-1}_{\mu\nu,k}}{\partial x^m} -2H^{(\nu_0)l-1}_{\mu\nu}\right)\ast G_{\nu_0}.
\end{array}
\end{equation}
The first crucial step is the convergence of the function $h^{(\nu_0)l}_{\mu\nu}$ as $l\uparrow \infty$. We obtain this convergence by a local contraction result in an appropriate function space. In order to choose this space, we first observe that it may be chosen weaker than expected. First note that for a classical solution of the Einstein evolution equation we have
\begin{equation}
g_{\mu\nu}\in C^{2,2}.
\end{equation}
We would expect that we have to construct $g^{(\nu_0)}_{\mu\nu}\in C^{2,2}$ accordingly. Nevertheless, it is sufficient to construct $g^{(\nu_0)}_{\mu\nu}$ in a subspace of $C^{1,2}$ first. In order to observe this note first that in the limit ($l\uparrow \infty$) of the iterated viscosity system we have the fixed point representation
\begin{equation}\label{repfix}
\begin{array}{ll}
 g^{(\nu_0)}_{\mu\nu}=g_{0\mu\nu}\ast_{sp} G_{\nu_0}+h^{(\nu_0)}_{\mu\nu} \ast G_{\nu_0},~\\
 \\
 g^{(\nu_0)}_{\mu\nu,k}=g_{0\mu\nu,k}\ast_{sp} G_{\nu_0}+\frac{\partial h^ {(\nu_0)}_{\mu\nu}}{\partial x^k}\ast G_{\nu_0}\\
 \\
  h^{(\nu_0)}_{\mu\nu}=h_{0\mu\nu}\ast_{sp} G_{\nu_0}\\
  \\
  -g^{(\nu_0)}_{00}\left(2g^{(\nu_0),0k}\frac{\partial h^{(\nu_0)}_{\mu\nu}}{\partial x^k}+g^{(\nu_0),km}\frac{\partial g^{(\nu_0)}_{\mu\nu,k}}{\partial x^m} -2H^{(\nu_0)}_{\mu\nu}\right)\ast G_{\nu_0}.
\end{array}
\end{equation}
The right side of (\ref{repfix}) has only spatial derivatives of the metric components $g^{(\nu_0)}_{\mu\nu}$ up to second order and first order spatial derivatives of the components $h_{\mu\nu}$, i.e., first order spatial derivatives of the first order time derivative of the metric components. We shall observe below, that for an appropriate choice of data and carefully chosen functions spaces regularity of these functions is preserved in the viscosity limit $\nu_0\downarrow 0$. Solutions by iterations of the viscosity system and their viscosity limit are naturally considered in a subspace of $C^{1,2}$ for the metric component functions $g_{\mu\nu}$ and in a subspace of $C^{0,1}$ for the metric components $h_{\mu\nu}$. Having constructed limits (iteration limit and viscosity limit) with $g_{\mu\nu}\in C^{1,2}\left((0,T)\times {\mathbb R}^n \right) $ and $h_{\mu\nu}\in C^{0,1}\left((0,T)\times {\mathbb R}^n \right)$ from the third equation of (\ref{harm2}) we get
\begin{equation}
\frac{\partial h_{\mu\nu}}{\partial t}=-g_{00}\left(2g^{0k}\frac{\partial h_{\mu\nu}}{\partial x^k}+g^{km}\frac{\partial g_{\mu\nu,k}}{\partial x^m} -2H_{\mu\nu}\right)\in C,
\end{equation}
where $C$ is the space of continuous functions. This implies that we have a classical solution.

In order to prepare this argument, we use the convolution rule and write the second equation in (\ref{repl}) in the form
\begin{equation}
g^{(\nu_0)l}_{\mu\nu,k}=g_{0\mu\nu,k}\ast_{sp} G_{\nu_0}+ h^ {(\nu_0)l-1}_{\mu\nu}\ast G_{\nu_0,k},
\end{equation}
avoiding the consideration of the convergence of spatial derivatives of $h_{\mu\nu}$, where we use later the fact that spatial first order derivatives of the Gaussian $G_{\nu_0,k}$ are locally integrable. Moreover, the convoluted metric functions and their derivatives are spatially Lipschitz in our construction.  
Note that it is essential to have convergence of the $h^{(\nu_0)l}_{\mu\nu}$ (as $l\uparrow \infty$) determined by the third equation.
We again use the convolution rule in order to avoid first order spatial derivatives of $h^{(\nu_0)l}_{\mu\nu}$ and second order spatial derivatives of the metric $g^{(\nu_0)l}_{\mu\nu}$ in the convoluted term.

We get for $l\geq 1$
\begin{equation}\label{hlij}
\begin{array}{ll}
h^{(\nu_0)l}_{\mu\nu}=h_{0\mu\nu}\ast_{sp} G_{\nu_0}-\left( g^{(\nu_0)l-1}_{00}2g^{(\nu_0)l-1,0k} h^{(\nu)l-1}_{\mu\nu}\right) \ast G_{\nu_0 ,k}\\
\\
+\left( g^{(\nu_0)l-1}_{00}2g^{(\nu_0)l-1,0k}\right)_{,k} h^{(\nu_0)l-1}_{\mu\nu}\ast G_{\nu_0 }+\left(  g^{(\nu_0)l-1}_{00} g^{(\nu_0)l-1,km}
 g^{(\nu_0)l-1}_{\mu\nu,k}\right) \ast G_{\nu_0,m} \\
\\
-\left( \left( g^{(\nu_0)l-1}_{00}  g^{(\nu_0)l-1,km}\right)_{,m}
g^{(\nu_0)l-1}_{\mu\nu,k}\right)  \ast G_{\nu_0} +2g^{(\nu_0)l-1}_{00} H^{(\nu_0)l-1}_{\mu\nu}\ast G_{\nu_0}.
\end{array}
\end{equation}
Note that by the use of the convolution rule all terms involve only first derivatives of the metric $g_{\mu\nu}$. 
Using (\ref{hlij}) we can compute the functions $h^{(\nu_0),l}_{\mu\nu},~0\leq \mu,\nu\leq n$ and consider the iteration limit $l\uparrow \infty$ and the viscosity limit $(\nu_0)\downarrow 0$. Next we consider the series $h^{(\nu_0)l}_{\mu\nu},~l\geq 1,~0\leq \mu,\nu\leq n$. In order to prove convergence in a strong function space we consider for $l\geq 2$ the functional series
\begin{equation}
h^{(\nu_0)l}_{\mu\nu}=h^{(\nu_0)1}_{\mu\nu}+\sum_{m=2}^l\delta h^{(\nu_0)m}_{\mu\nu},
\end{equation}
 where
\begin{equation}
\delta h^{(\nu_0)m}_{\mu\nu}=h^{(\nu_0)m}_{\mu\nu}-h^{(\nu_0)m-1}_{\mu\nu},
\end{equation}
and where $h^{(\nu_0)1}_{\mu\nu}$ is defined by (\ref{hlij}) applied to the data, where for $l=1$ we have  
$$
\begin{array}{ll}
g^{(\nu_0)l-1}_{\mu\nu}(0,.)=g^{(\nu_0)0}_{\mu\nu}(0,.)=g_{0ij},~g^{(\nu_0)l-1}_{\mu\nu,t}(0,.)=g^{(\nu_0)0}_{\mu\nu,t}(0,.)=h_{0ij},\\
\\
g^{(\nu_0)l-1}_{\mu\nu,k}(0,.)=g^{(\nu_0)0}_{\mu\nu,k}(0,.)=g_{0\mu\nu,k}.
\end{array}
$$

We have
\begin{equation}
\begin{array}{ll}
 \delta h^{(\nu_0)l}_{\mu\nu}=-\left( g^{(\nu_0)l-1}_{00}2g^{(\nu_0)l-1,0k} h^{(\nu)l-1}_{\mu\nu}\right) \ast G_{\nu_0 ,k}\\
\\
+\left( g^{(\nu_0)l-1}_{00}2g^{(\nu_0)l-1,0k}\right)_{,k} h^{(\nu_0)l-1}_{\mu\nu}\ast G_{\nu_0 }+\left(  g^{(\nu_0)l-1}_{00} g^{(\nu_0)l-1,km}
 g^{(\nu_0)l-1}_{\mu\nu,k}\right) \ast G_{\nu_0,m} \\
\\
-\left( \left( g^{(\nu_0)l-1}_{00}  g^{(\nu_0)l-1,km}\right)_{,m}
g^{(\nu_0)l-1}_{\mu\nu,k}\right)  \ast G_{\nu_0} +2g^{(\nu_0)l-1}_{00} H^{(\nu_0)l-1}_{\mu\nu}\ast G_{\nu_0}\\
\\
+\left( g^{(\nu_0)l-2}_{00}2g^{(\nu_0)l-2,0k} h^{(\nu)l-2}_{\mu\nu}\right) \ast G_{\nu_0 ,k}\\
\\
-\left( g^{(\nu_0)l-2}_{00}2g^{(\nu_0)l-2,0k}\right)_{,k} h^{(\nu_0)l-2}_{\mu\nu}\ast G_{\nu_0 }-\left(  g^{(\nu_0)l-2}_{00} g^{(\nu_0)l-2,km}
 g^{(\nu_0)l-2}_{\mu\nu,k}\right) \ast G_{\nu_0,m} \\
\\
+\left( \left( g^{(\nu_0)l-2}_{00}  g^{(\nu_0)l-2,km}\right)_{,m}
g^{(\nu_0)l-2}_{\mu\nu,k}\right)  \ast G_{\nu_0} -2g^{(\nu_0)l-2}_{00} H^{(\nu_0)l-2}_{\mu\nu}\ast G_{\nu_0}.
\end{array}
\end{equation}
Interpolation, i.e. subtraction and addition of mixed terms $$-\left( g^{(\nu)l-1}_{00}2g^{(\nu)l-1,0k} h^{(\nu)l-2}_{\mu\nu}\right) \ast G_{\nu ,k} \mbox{ etc.}$$ leads to a functional
\begin{equation}
\left(\delta g^{(\nu_0)l}_{\mu\nu},\delta g^{(\nu_0)l}_{\mu\nu,k},\delta h^{(\nu_0)l}_{\mu\nu} \right)=F\left(\delta g^{(\nu_0)l-1}_{\mu\nu},\delta g^{(\nu_0)l-1}_{\mu\nu,k},\delta h^{(\nu_0)l-1}_{\mu\nu},p \right) 
\end{equation}
with a 'parameter vector'
\begin{equation}
p=\left( g^{(\nu_0)l-1}_{\mu\nu},g^{(\nu_0)l-1}_{\mu\nu,k}, h^{(\nu_0)l-1}_{\mu\nu},g^{(\nu_0)l-2}_{\mu\nu}, g^{(\nu_0)l-2}_{\mu\nu,k},h^{(\nu_0)l-2}_{\mu\nu}\right) 
\end{equation}
of functions at the start of iteration step $l$. Symmetry allows us to consider indices $\mu\leq \nu$, which reduces the function space, but this does not really matter. Next we determine an appropriate function space. As we want $g_{\mu\nu}\in C^{2}$ outside the origin and where $g_{\mu\nu}$ is well defined (here $C^2$ is the function space of twice differentiable functions with continuous derivatives), our choice of a functions space close to $C^{1,2}$ for an iteration for $g_{\mu\nu}$ (resp. $g^{(\nu_0)}_{\mu\nu}$) via approximations $g^{(\nu_0),l}_{\mu\nu}$ looks rather weak, but it is fitting as the time derivative of $h_{\mu\nu}$ is defined in terms of spatial derivatives up to second order of $g_{\mu\nu}$ and spatial derivatives up to first order of $h_{\mu\nu}$. 
For space-time dimension $n+1$ we have essentially $(n+1)(n+2)/2$ metric functional increments $\delta g^{(\nu_0)l}_{\mu\nu}=g^{(\nu_0)l}_{\mu\nu}-g^{(\nu_0)l-1}_{\mu\nu}$, $n(n+1)(n+2)/2$ metric functional increments $\delta g^{(\nu_0)l}_{\mu\nu,k}=g^{(\nu_0)l}_{\mu\nu,k}-g^{(\nu_0)l-1}_{\mu\nu,k}$ and $(n+1)(n+2)/2$ increments $\delta h^{(\nu_0)l}_{\mu\nu}$. Accordingly, we define
\begin{equation}
G^{(\nu_0)}_l=(\delta g^{(\nu_0)l}_{\mu})_{0\leq \mu\leq \nu\leq n},~G^{(\nu_0)}_{l1}=(\delta g^{(\nu_0)l}_{ij})_{0\leq \mu\leq \nu\leq n,~1\leq k\leq n},
\end{equation}
and
\begin{equation}
{\cal H}^{(\nu)}_l=(\delta h^{(\nu_0)l}_{\mu\nu})_{0\leq \mu\leq \nu\leq n}.
\end{equation}
For the sake of abbreviation we write
\begin{equation}
\Omega_T:=(0,T)\times {\mathbb R}^n.
\end{equation}
Then with $d_g=(n+1)(n+2)/2$, $d_{g1}=n(n+1)(n+2)/2$,  $d_h=(n+1)(n+2)/2$ (we shall choose  $m=1$ and $l=2$ later, but for more regular data we could adopt $m$ and $l$- therefore we use a more general notation in the following) we have a 
functional
\begin{equation}\label{F}
\begin{array}{ll}
F:\left[ C_{H^2,\circ}^{m,l}\left(\Omega_T \right)\right]^{d_g}\times \left[ C_{H^2,\circ}^{m,l-1}\left(\Omega_T \right)\right]^{d_{g1}}\times \left[ C_{H^2,\circ}^{m-1,l-1}\left(\Omega_T \right)\right]^{d_{h}}\\
\\
\rightarrow \left[ C_{H^2,\circ}^{m,l}\left(\Omega_T \right)\right]^{d_g}\times \left[ C_{H^2,\circ}^{m,l}\left(\Omega_T \right)\right]^{d_{g1}}\times \left[ C_{H^2,\circ}^{m,l}\left(\Omega_T \right)\right]^{d_{h}},\\
\\
\left(G^{(\nu_0)}_l,G^{(\nu_0)}_{l1},{\cal H}^{(\nu_0)}_l\right)^T
=F\left(G^{(\nu_0)}_{l-1},G^{(\nu_0)}_{(l-1)1},{\cal H}^{(\nu_0)}_{l-1}\right) 
\end{array}
\end{equation}
with the function space $C_{H^2,\circ}^{m,l}\left(\Omega_T \right)$ which is defined
as
\begin{equation}
C_{H^2,\circ}^{m,l}\left(\Omega_T \right)=\left\lbrace f\in C_{\circ}^{m,l}\left(\Omega_T \right)|\forall t\in [0,T]:~f(t,.)\in H^2  \right\rbrace 
\end{equation}
along with the functions space  \begin{equation}\label{incremomega}
C_{\circ}^{m,l}\left(\Omega_T\right):=\left\lbrace f:\Omega_T\rightarrow {\mathbb R}|f\in C^{m,l}~\&~\forall x~f(0,x)=0 \right\rbrace.
\end{equation} 

Note that by the use of the the convolution rule we have obtained that representations of  the increments have only first order derivatives of the metric functions $g^{(\nu_0)p}_{ij}$ for some $p\geq 0$. Some of the related Gaussian $G_{\nu_0}$-terms then get first order derivatives, where these Gaussians have local standard $L^1$-estimates in the time interval $(0,T]$ (open at $0$) and on a ball around a fixed spatial argument $x$ (cf. also the proof of Lemma 3.2 below.  Similarly, for the $h^{(\nu_0),l}_{\mu\nu}$-terms. Outside the ball around a fixed argument $x$ we surely have $L^1$-estimates of the Gaussian such that we can apply Young inequalities in order to get contraction of $F$ on same time interval, i.e., we have
\begin{lem}\label{mainlem}
There is a time horizon $T>0$ such that the map $F$ is a contraction on the function space $$ \left[ C_{H^2,\circ}^{m,l}\left(\overline{D_T} \right)\right]^{d_g}\times \left[ C_{H^2,\circ}^{m,l-1}\left(\overline{D_T} \right)\right]^{d_{g1}}\times \left[ C_{H^2,\circ}^{m-1,l-1}\left(\overline{D_T} \right)\right]^{d_{h}}$$
with a contraction constant $c\in (0,1)$
\end{lem}

The latter contraction result leads to the pointwise limit
\begin{equation}
g^{(\nu_0)}_{\mu\nu}=\lim_{l\uparrow \infty}g^{(\nu_0)l}_{\mu\nu}\in C^{1,2}\left(\left(0,T\right],{\mathbb R}  \right),
\end{equation}
where 
\begin{equation}
g^{(\nu_0)}_{\mu\nu}(0,.)\in C^{1,\delta}\left({\mathbb R}^n\right)\cap H^2.
\end{equation}
We denote
\begin{equation}
h^{(\nu_0)}_{\mu\nu}=\lim_{l\uparrow \infty}h^{(\nu_0)l}_{\mu\nu}
\end{equation}
for all $1\leq \mu,\nu\leq n$. Accordingly we write for all $1\leq i,j\leq n$
\begin{equation}
H^{(\nu_0)}_{\mu\nu}=\lim_{l\uparrow \infty}H^{(\nu_0)l}_{\mu\nu}
\end{equation}
where we recall that we have
\begin{equation}
g^{(\nu_0)}_{\mu\nu}(0,.)=g^{(\nu_0)}_{0ij},~g^{(\nu_0)}_{\mu\nu,t}(0,.)=h^{(\nu_0)}_{0\mu\nu},~g^{(\nu)}_{\mu\nu,k}(0,.)=g^{(\nu)}_{0\mu\nu,k}
\end{equation}
for the initial data (which do not depend on the iteration index $l\geq -1$. 
We observe that the contraction is essentially independent of the viscosity parameter $\nu_0$ (for small $\nu_0 >0$). This is not suprising as the density function $(t,y)\rightarrow G_{\nu_0}(t,y)=\frac{1}{\sqrt{4\pi \nu_0 t}^n}\exp\left(-\frac{|y|^2}{4\nu_0 t} \right) $ is integrable on the domain $\left((0,T]\times {\mathbb R}^n\right)$, where for $\delta\in (0,1)$ we have
\begin{equation}
{\big |}G_{\nu_0,i}(t,y){\big |}={\Big |}\frac{-2y_i}{4\nu_0 t}G_{\nu_0}{\Big |}\leq {\Big |}\frac{C}{|y|}G_{\nu_0}{\Big |}\leq \frac{C}{(\nu_0 t)^{\delta}|y|^{n+1-2\delta}}
\end{equation}
with $C=\sup_{z>0}\frac{1}{2}z^2\exp\left(-\frac{1}{4}z^2\right)$ is locally integrable for $n\geq 2$ (and $\delta >0.5$) such that the Gaussian is then easily seen to be globally integrable. We have  
\begin{lem}
The contraction constant $c\in (0,1)$ of Lemma \ref{mainlem} can be chosen independently of the viscosity constant $\nu_0 >0$.
\end{lem}

\begin{proof}
We have observed that the essential recursive functionals in (\ref{hlij}) can be written such that the convoluted terms on the right side involve only products of metric components $g_{ij}$, their inverses and  first order spatial derivatives of such metric components or products of such metric components. Here we have rewritten the terms where a second order spatial derivative of a metric tensor component appears.  
For data $g^{(\nu_0)}_{\mu\nu}(0,.)$ or data entries of the inverse (essentially as in (\ref{dataf})) we have Lipschitz continuity, but only local Lipschitz continuity of the first spatial derivatives. Nevertheless we can prove the existence of a regular solution branch using classical solution representations of local time solutions, where in these representations we approximate all first order derivatives $g_{ij,k}$ by approximative convolutions $g_{ij}\ast G_{\nu,k}$, since the latter convolutions have Lipschitz continuous upper bounds (cf. below). Then we can consider viscosity limits. 

Representations of (approxmations) of solution functions in terms of convolutions with Lipschitz continuous functions or Lipschitz continuous upper bounds with  first order spatial derivatives of the Gaussian have the advantage that symmetry relations of the form
\begin{equation}\label{visclim}
\begin{array}{ll}
{\Big |}f\ast_{sp}G_{\nu_0,i}(s,x){\Big |}= {\Big |}\int_{{\mathbb R}^n}f(x-y)\frac{2y_i}{4\nu s\sqrt{4\pi \nu_0 s}^n}\exp\left(-\frac{|y|^2}{4\nu_0 s}\right) dy{\Big |}\\ 
 \\
 \leq \int_{\left\lbrace y|y\in {\mathbb R}^D, y_i\geq 0\right\rbrace }{\Big |}f_x(y)-f_x(y^{i,-}){\Big |}\frac{2|y_i|}{4\nu_0 s\sqrt{4\pi \nu s}^n}\exp\left(-\frac{|y|^2}{4\nu_0 s}\right) dy\\
 \\
 \leq \int_{\left\lbrace y|y\in {\mathbb R}^D, y_i\geq 0\right\rbrace }L{\Big |}2y_i{\Big |}\frac{2|y_i|}{4\nu_0 s\sqrt{4\pi \nu s}^n}\exp\left(-\frac{|y|^2}{4\nu_0 s}\right) dy\\
 \\
 =\int_{\left\lbrace y'|y'\in {\mathbb R}^D, y'_i\geq 0\right\rbrace }L{\Big |}2\sqrt{\nu_0}y'_i{\Big |}\frac{2|\sqrt{\nu_0}y'_i|}{4\pi \nu_0 s\sqrt{4\pi  s}^n}\exp\left(-\frac{|y'|^2}{4 s}\right) dy'\\
 \\
 =\int_{\left\lbrace y'|y'\in {\mathbb R}^D, y'_i\geq 0\right\rbrace }L{\Big |}2y'_i{\Big |}\frac{|2y'_i|}{4 s\sqrt{4\pi  s}^n}\exp\left(-\frac{|y'|^2}{4 s}\right) dy'\\
 \\
 =\int_{\left\lbrace y'|y'\in {\mathbb R}^D, y'_i\geq 0\right\rbrace }L\frac{4(y'_i)^2}{4 s\sqrt{4\pi  s}^n}\exp\left(-\frac{|y'|^2}{4 s}\right) dy'\leq 4LC
\end{array}
 \end{equation}
(with $y=\sqrt{\nu_0}y'$ and for some Lipschitz constant $L$ and a finite constant $C>0$ which is independent of $\nu_0$) can be used. Here, $y^{i,-}$ has the components $(y^{i,-}_j,~1\leq j\leq n$ with
\begin{equation}
y^{i,-}_j=y^{i}_j,~\mbox{ if }j\neq i~\mbox{ and }~y^{i,-}_i=y_i~\mbox{ for }~j=i.
\end{equation}
This estimate can alway be used in our situation as we have Lipschitz continuous upper bounds. In this contexts note that for data constructions as in (\ref{dataf}))
we have for $\delta >0$ small
\begin{equation}
\begin{array}{ll}
{\Big |}\left( z^{3}\sin\left(\frac{1}{z^{\alpha_0}} \right)\phi_\delta\ast_{sp}G_{\nu,1}\right) (t,x){\Big |}\\
\\
\leq 2{\big |}\int_{z>0}\left( |z-y|^{3} \phi_\delta(z-y)G_{\nu_0,1}(t,y)\right) (t,x){\big |}\leq C|z|\phi_\delta.
\end{array}
\end{equation}

Alternatively, we may use the fact that there is an $L^1\left((0,T)\times {\mathbb R}^n \right) $ upper bound of Gaussian $G_{\nu_0}$ and its first order derivatives $G_{\nu_0,k}, 1\leq k\leq n$ which are independent of the viscosity $\nu_0 >0$. 
First, for $\Delta x=x-y$ and $\Delta s=4\nu_0\Delta t$ we have for the essential factor of the Gaussian for some $C>0$ and $\delta >0$ (where $z=\frac{\Delta x}{\sqrt{\Delta s}}$)
\begin{equation}
\begin{array}{ll}
\frac{1}{\sqrt{\Delta s}^n}\exp\left(-\frac{\Delta x^2}{\Delta s} \right) = \left( \frac{\Delta x^2}{\Delta s}\right)^{\frac{n}{2}-\delta}\frac{1}{\Delta s^{\delta}|\Delta x|^{n-2\delta}} \exp\left(-\frac{\Delta x^2}{\Delta s} \right)\\
\\
\leq\frac{1}{\Delta s^{\delta}|\Delta x|^{n-2\delta}}\sup_{z\in {\mathbb R}}\left( z^2\right)^{\frac{n}{2}-\delta}\exp\left(-z^2 \right)=\frac{C}{\Delta s^{\delta}|\Delta x|^{n-2\delta}}.
\end{array}
\end{equation}
Similarly, for the first order spatial derivatives of the (essential factor of the) Gaussian we have for $\delta\in \left(\frac{1}{2},1 \right)$ in a ball $B_{R}(x)$ of radius $R>0$ the estimate 
\begin{equation}
\begin{array}{ll}
\left( \frac{1}{\sqrt{\Delta s}^n}\exp\left(-\frac{\Delta x^2}{\Delta s} \right) \right)_{,i}
\leq\frac{C}{\Delta s^{\delta}|\Delta x|^{n+1-2\delta}}.
\end{array}
\end{equation}
Hence for $T>0$ and some $c'>0$ we have
\begin{equation}
\int_0^T\int_{B_R(x)}\frac{d\Delta td\Delta x}{\Delta s^{\delta}|\Delta x|^{n+1-2\delta}}\leq \frac{c'}{\nu_0^{\delta}}|T|^{1-\delta}R^{2\delta -1}.
\end{equation}
This means that for $R=\nu_0^{2}$ we have $R^{2\delta-1}=\nu_0^{4\delta -2}$ and $\delta \in \left(\frac{2}{3},1 \right)$ we have
\begin{equation}
\frac{c'}{\nu_0^{\delta}}|T|^{1-\delta}\nu_0^{4\delta -2}= c'|T|^{1-\delta}\nu_0^{3\delta -2}\downarrow 0
\end{equation}
as $\nu_0 \downarrow 0$. Furthermore, for the complement ${\mathbb R}^n\setminus B_{R}(x)$ with $R=\nu_0^2$ we have for some finite constant $c>0$ the essential estimate
\begin{equation}
\begin{array}{ll}
{\Big |}\int_{|\Delta x|\geq \nu_0^{2}}\left(-\frac{\Delta x_k}{\nu_0 \Delta t} \right) \frac{1}{\sqrt{\nu_0\Delta t}^n}\exp\left(-\frac{\Delta x^2}{\nu_0\Delta t} \right)dy{\Big |}\\
\\
\leq {\Big |}\int_{r\geq \nu_0^{2}}\left(\frac{r}{\nu_0 \Delta t} \right) \frac{c}{\sqrt{\nu\Delta t}^n}\exp\left(-\frac{r^2}{\nu_0\Delta t} \right)r^{n-1}dr{\Big |}\\
\\
\leq {\Big |}\left(\frac{1}{2} \right) \frac{c}{\sqrt{\nu_0\Delta t}^n}\exp\left(-\frac{r^2}{\nu_0\Delta t} \right)r^{n-1}{\Big |}^{\infty}_{\nu^2_0}\\
\\
+{\Big |}\int_{r\geq \nu_0^{2}}\left(\frac{n-1}{2} \right) \frac{c}{\sqrt{\nu_0\Delta t}^n}\exp\left(-\frac{r^2}{\nu_0\Delta t} \right)r^{n-2}dr{\Big |}\downarrow 0~\mbox{as}~\nu_0\downarrow 0.
\end{array}
\end{equation}
Here we have absorbed the factor $4$ in the time variable implicitly.   
It follows that we have a uniform bound
\begin{equation}
\begin{array}{ll}
\sup_{\nu_0 >0}\left( {\big |}G_{\nu_0}{\big |}_{L^1\left((0,T)\times {\mathbb R}^n \right) }
+\sum_{k=1}^n{\big |}G_{\nu_0,k}{\big |}_{L^1\left((0,T)\times {\mathbb R}^n \right) }\right) \leq C
\end{array}
\end{equation}
for some $C>0$. 
\end{proof}
Using local contraction the viscosity limit ($\nu_0\downarrow 0$) and the iteration limit $l\uparrow \infty$ can be considered at the same time. We choose a sequence $\nu_l,~l\geq 1$ with $\lim_{l\uparrow \infty}\nu_l=0$ and consider the functional series $g^{(\nu_l),l}_{\mu\nu},~l\leq 2$, where we consider the representations
\begin{equation}
g^{(\nu_l),l}_{\mu\nu}=g^{(\nu_l),0}_{\mu\nu}+\sum_{p=2}^{l}\delta g^{(\nu_l),p}_{\mu\nu}.
\end{equation}
The componentwise differentiation (up to second order) of the limit $g_{\mu\nu}:=\lim_{l\uparrow \infty }g^{(\nu_l),l}_{\mu\nu}$ of this functional series is a delicate matter. However, we proceed as follows. First we consider functions on the domain $D_T=\left[0,T\right] \times \left]-\frac{\pi}{2},\frac{\pi}{2}\right[^n$
\begin{equation}
\begin{array}{ll}
g^{(\nu_l),l}_{\mu\nu}(.,\tan(.)):D_T\rightarrow {\mathbb R},~~
g^{(\nu_l),l}_{\mu\nu,k}(., \tan(.)):D_T \rightarrow {\mathbb R}\\
\\
g^{(\nu_l),l}_{\mu\nu,k,l}(., \tan(.)):D_T \rightarrow {\mathbb R}
\end{array}
\end{equation}
for indices $0\leq \mu,\nu\leq n$ and 'spatial indices' $1\leq k,l\leq n$. Here we denote $$\tan(x)=(\tan(x^1),\cdots ,\tan(x^n))^T.$$ Since $g^{(\nu_l),l}_{\mu\nu}(t,.)\in H^2\cap C^2$for all $t\in (0,T]$ for all these functions we have
\begin{equation}
\begin{array}{ll}
\forall t\in [0,T]:~g^{(\nu_l),l}_{\mu\nu}\left(t,\tan(-\frac{\pi}{2})\right) =g^{(\nu_l),l}_{\mu\nu}(t, \tan(\frac{\pi}{2})=0\\
\\
\forall t\in [0,T]:g^{(\nu_l),l}_{\mu\nu,k}\left(t,\tan(-\frac{\pi}{2})\right)=g^{(\nu_l),l}_{\mu\nu,k}\left(t,\tan(\frac{\pi}{2})\right)=0\\
\\
\forall t\in [0,T]:~g^{(\nu_l),l}_{\mu\nu,k,l}\left(t,\tan(-\frac{\pi}{2})\right)=g^{(\nu_l),l}_{\mu\nu,k,l}\left(t,\tan(\frac{\pi}{2})\right)=0.
\end{array}
\end{equation}
Hence we have periodic extensions
\begin{equation}
\begin{array}{ll}
d^{(\nu_l),l}_{\mu\nu}:D_T\rightarrow {\mathbb R},~~d^{\nu_l,l}_{\mu\nu}|_{D_T}=g^{(\nu_l),l}_{\mu\nu}(., \tan(.)),~d^{(\nu_l),l}_{\mu\nu}\left(t,-\frac{\pi}{2}\right)=d^{\nu,l}_{\mu\nu}\left(t, \frac{\pi}{2}\right),\\
\\
e^{(\nu_l),l}_{\mu\nu k}:D_T\rightarrow {\mathbb R},~~e^{(\nu_l),l}_{\mu\nu k}|_{D_T}=g^{(\nu),l}_{ij,k}(., \tan(.)),~e^{\nu,l}_{\mu\nu k}\left(t, -\frac{\pi}{2}\right)=e^{(\nu_l),l}_{\mu\nu k}\left(t, \frac{\pi}{2}\right),\\
\\
f^{(\nu_l),l}_{\mu\nu kl}:D_T \rightarrow {\mathbb R},~~f^{(\nu_l),l}_{\mu\nu kl}|_{D_T}=g^{(\nu)_l,l}_{ij,k,l}\left( ., \tan(.)\right) ,~f^{(\nu_l),l}_{\mu\nu kl}\left(t, -\frac{\pi}{2}\right)=f^{\nu,l}_{\mu\nu kl}\left(t, \frac{\pi}{2}\right).
\end{array}
\end{equation}
Note that we can recover the viscosity and iteration limit  $g_{ij}$ ($\nu_l\downarrow 0$, $l\uparrow \infty$) and its derivatives up to second order from the limits of the functions $d^{(\nu_l),l}_{\mu\nu}$ and $e^{(\nu_l),l}_{\mu\nu k},f^{(\nu_l),l}_{\mu\nu kl}$ ($\nu_l\downarrow 0$) respectively.  
We denote the standard closure of the domain $D_T$ by $\overline{D_T}$. We are interested in the strong convergence of the increments
\begin{equation}\label{increm}
\begin{array}{ll}
d^{(\nu_l),l}_{\circ,\mu\nu}:=d^{(\nu_l),l}_{\mu\nu}-d^{(\nu_l),l}_{\mu\nu}(0,.),~e^{(\nu_l),l}_{\circ,\mu\nu k}:=e^{(\nu_l),l}_{\mu\nu k}-e^{(\nu_l),l}_{\mu\nu k}(0,.)\\
\\
f^{(\nu_l),l}_{\circ,\mu\nu k}:=f^{(\nu_l),l}_{\mu\nu k}-f^{(\nu_l),l}_{\mu\nu k}(0,.).
\end{array}
\end{equation}
For $m=1$ and $l=2$ these increments $d^{(\nu_l),l}_{\circ,\mu\nu}$ are located in the appropriate function space
\begin{equation}\label{incremf}
C_{\circ}^{m,l}\left(\overline{D_T} \right):=\left\lbrace f:\overline{D_T}\rightarrow {\mathbb R}|f\in C^{m,l}~\&~\forall x~f(0,x)=0 \right\rbrace,
\end{equation}
and the increments $e^{(\nu_l),l}_{\circ,\mu\nu k}$ and $f^{(\nu_l),l}_{\circ,\mu\nu k}$ are located in the function spaces $C_{\circ}^{m,l-1}\left(\overline{D_T}\right)$ and $C_{\circ}^{m-1,l-1}\left(\overline{D_T} \right)$ respectively.
For functional series in this function space we may use the following classical result
\begin{thm}
Consider a functional series $F_m:=\sum_{n=1}^mf_n,~m\geq 1$ with $f_n\in C_{\circ}^{0,1}\left(\overline{D_T} \right)$. Assume that $F_m(c),~m\geq 1$ converges for fixed $c\in D_T$, and assume that the first order spatial derivative functional series $F_{m,i}:=\sum_{n=1}^mf_{n,i},~m\geq 1$ converge uniformly in $\overline{D_T}$. Then the functional series $F_m,¸ m\geq 1$ converges uniformly to a function $F=\lim_{m\uparrow \infty}\sum_{n=1}^mf_n$, and such that for all $(t,x)\in {\overline D_T}$ 
\begin{equation}
F_{,i}(t,x)=\lim_{m\uparrow\infty}\sum_{n=1}^mf_{n,i}(t,x)~\mbox{holds.}
\end{equation}
\end{thm}

\begin{lem}
There is a sequence $(\nu_l)_{l\geq 1}$ with $\lim_{l\uparrow \infty }\nu_l=0$ such that the limits of the functional increments in (\ref{increm}) are in the function space (\ref{incremf}), i.e.,
\begin{equation}
\begin{array}{ll}
d^0:=\lim_{l\uparrow \infty}d^{\nu_l,l}_{\circ,\mu\nu} ,~e^0:=\lim_{l\uparrow \infty}e^{\nu_l,l}_{\circ,\mu\nu k},~
f^0:=\lim_{l\uparrow \infty}f^{\nu_l,l}_{\circ,\mu\nu k}\in C_{\circ}^{0,1}\left(\overline{D_T} \right)
\end{array}
\end{equation}
\end{lem}

Hence,
\begin{equation}\label{limg}
\begin{array}{ll}
 g_{\mu\nu}=\lim_{l\uparrow \infty}g^{(\nu_l),l}_{\mu\nu}\in C^{0,2}(D_T),~
g_{\mu\nu,k}=\lim_{l\uparrow \infty} g^{(\nu_l),l}_{\mu\nu,k}\in C^{0,1}(D_T),
\end{array}
\end{equation}
and, hence, for the first order time derivative, we get 
\begin{equation}\label{limh}
\begin{array}{ll}
  h_{\mu\nu}=\lim_{l\uparrow \infty}h^{(\nu_l),l}_{ij}\in C^{0,1}(D_T).
\end{array}
\end{equation}
We have to check that this limit is indeed a solution. First observe that we may consider the iteration limit ($l\uparrow \infty$) first.
For $\nu_0>0$ let $g^{\nu_0}_{\mu\nu},~0\leq \mu,\nu\leq n$ be a fixed point in (\ref{repfix}) with  $g^{(\nu_0)}_{\mu\nu}(t,.)\in C^2$.
Plugging $g^{(\nu_0)}_{\mu\nu}$ into the harmonic field equation  and using the local contraction result we observe that the limit $g_{\mu\nu}=\lim_{\nu_0\downarrow 0}=g^{(\nu_0)}_{\mu\nu},~g_{\mu\nu,k}=\lim_{\nu_0\downarrow 0}g^{(\nu_0)}_{\mu\nu,k},~\lim_{\nu_0\downarrow 0}h^{(\nu_0)}_{\mu\nu}$ in (\ref{limg}) and in (\ref{limh}) satisfies the harmonic field equation.  Next abbreviate
\begin{equation}
\begin{array}{ll}
\mathbf{f}=\left(f_1,\cdots,f_M\right)\\
\\
=\left(\left( g^{(\nu_0)}_{0\mu\nu}\right)_{0\leq \mu,\nu\leq n},~\left( g^{(\nu_0)}_{0\mu\nu,k}\right)_{0\leq \mu,\nu\leq n,~1\leq k\leq n},~\left( h^{(\nu_0)}_{0\mu\nu}\right)_{0\leq \mu,\nu\leq n}\right)^T,  
\end{array}
\end{equation}
and
\begin{equation}
\begin{array}{ll}
\mathbf{v}^{(\nu_0)}=\left(v^{(\nu_0)}_1,\cdots,v^{(\nu_0)}_M\right)\\
\\
=\left(\left( g^{(\nu_0)}_{\mu\nu}\right)_{0\leq \mu,\nu\leq n},~\left( g^{(\nu_0)}_{\mu,\nu,k}\right)_{0\leq \mu,\nu\leq n,1\leq k\leq n},~\left( h^{(\nu_0)}_{\mu\nu}\right)_{0\leq \mu\nu\leq n}\right)^T, 
\end{array}
\end{equation}
and 
\begin{equation}
F^*(\mathbf{v}^{\nu_0})=\left(F^*_1\left( \left( h^{(\nu_0)}_{\mu\nu}\right)_{0\leq \mu,\nu\leq n}\right), F^*_2\left( \left( h^{(\nu_0)}_{ij,k}\right)_{0\leq \mu,\nu\leq n,1\leq k\leq n}\right), F^*_3\left( \left( \mathbf{v}^{(\nu_0)}\right)  \right)\right)^T,
\end{equation}
where
\begin{equation}
F^*_1\left( \left( h^{(\nu_0)}_{\mu\nu}\right)_{0\leq \mu,\nu\leq n}\right)=\left( h^{(\nu_0)}_{\mu\nu}\right)_{0\leq \mu,\nu\leq n}^T,
\end{equation}
\begin{equation}
F^*_2\left( \left( h^{(\nu_0)}_{\mu\nu,k}\right)_{0\leq \mu,\nu\leq n}\right)=\left( h^{(\nu_0)}_{\mu\nu,k}\right)_{0\leq \mu,\nu\leq n,~1\leq k\leq n}^T,
\end{equation}
and
\begin{equation}
F^*_3 \left( \mathbf{v}^{(\nu_0)}\right)=\left( 
-g^{(\nu_0)}_{00}\left(2g^{(\nu_0)0k}\frac{\partial h^{(\nu_0)}_{\mu\nu}}{\partial x^k}+g^{(\nu_0)km}\frac{\partial g^{(\nu_0)}_{\mu\nu,k}}{\partial x^m} -2H^{(\nu_0)}_{\mu\nu}\right)_{0\leq \mu,\nu\leq n}\right)^T 
\end{equation}
with the obvious identifications (strictly speaking also with some identifications of tuples of tuples $(f_i)_{1\leq i\leq M}$ and their entries). Note that we have imposed no index here, therefore the superscript $*$ at the symbol $F^*$ in order to indicate the difference to $F$.
The equation in (\ref{harm3}) may then be abbreviated as
\begin{equation}
\mathbf{v}^{(\nu_0)}_{,t}-\nu_0\Delta \mathbf{v}^{(\nu_0)}=F^*(\mathbf{v}^{(\nu_0)}),~\mathbf{v}^{(\nu_0)}(0,.)=\mathbf{f}
\end{equation}
where in the limit $\nu\downarrow 0$ the function $\mathbf{v}:=\mathbf{v}^{0}:=\lim_{\nu_0\downarrow 0}\mathbf{v}^{(\nu_0)}$ satisfies the equation
\begin{equation}\label{abbr}
\mathbf{v}_{,t}=F^*(\mathbf{v}).
\end{equation}
For $\mathbf{v}^{\nu,f}:=\mathbf{v}^{(\nu_0)}-\mathbf{f}$ we have $\mathbf{v}^{(\nu_0),f}(0,.)\equiv 0$ and
\begin{equation}
\mathbf{v}^{(\nu_0),f}_{,t}-\nu_0\Delta \mathbf{v}^{(\nu_0),f}=\nu_0 \Delta \mathbf{f}+F^*(\mathbf{v}^{(\nu_0),f}+\mathbf{f}).
\end{equation}
This leads to the representation
\begin{equation}
\mathbf{v}^{(\nu_0),f}=\nu_0 \Delta \mathbf{f}
\ast G_{\nu_0}+
F^*(\mathbf{v}^{(\nu_0),f}+\mathbf{f})\ast G_{\nu_0},
\end{equation}
where the convolution is understood componentwise. The latter statement has a classical interpretation only for smoothed initial data $\mathbf{f}\ast G_{\nu_0}$. However, we can rewrite in terms of first order derivative, i.e, we have
\begin{equation}
\mathbf{v}^{(\nu_0),f}=\nu_0 \sum_{i=1}^n \mathbf{f}_{,i}
\ast G_{\nu_0,i}+
F^*(\mathbf{v}^{(\nu_0),f}+\mathbf{f})\ast G_{\nu_0},
\end{equation}
where the derivative $_{,i}$ is understood componentwise of course, i.e., 
\begin{equation}
\div\mathbf{f}=\left(f_{1,i},\cdots,f_{n,i} \right)^T. 
\end{equation}
Hence, we have
\begin{equation}\label{last}
\mathbf{v}^f=\lim_{\nu_0\downarrow 0}\mathbf{v}^{(\nu_0),f}=\lim_{\nu_0\downarrow 0}\nu_0\sum_i \mathbf{f}_{,i}\ast G_{\nu_0,i}+\lim_{\nu_0\downarrow 0}F^*(\mathbf{v}^{(\nu_0),f}+\mathbf{f})\ast G_{\nu_0}.
\end{equation}
According to our analysis of the Gaussian above the first term on the rightside of the latter equation cancels, and, using continuity of $F$, we observe that the components of $\mathbf{v}^{(\nu_0),f}$ converge in the function space (\ref{incremf}) as $(\nu_0)\downarrow 0$. Considering the time derivative of the last term on the right side of equation (\ref{last}) for fixed $\nu_0>0$ and then going to the limit $\nu_0\downarrow 0$ we observe that the limit function $\mathbf{v}^f$ satisfies
\begin{equation}
\mathbf{v}^f_{,t}=F^*(\mathbf{v}^{f}+\mathbf{f}),
\end{equation}
which is equivalent of (\ref{abbr}). The singularity is generically $L^p$-stable by construction.
Finally we rmark that the provious considerrations can be extended to the case of admissable data with $c_1,c_2>0$ straightforwardly.
The previous argument leads to a fixed point viscosity limit $\nu_0\downarrow 0$ of the third equation in \ref{harm2} ofthe form
\begin{equation}\label{harm2end}
 \begin{array}{ll}
 h_{\mu\nu}={\Bigg (}-g^{(\nu_0)}_{00}\left(2g^{(\nu_0)0k}\frac{\partial h^{(\nu_0)}_{\mu\nu}}{\partial x^k}+g^{(\nu_0)km}\frac{\partial g^{(\nu_0)}_{\mu\nu,k}}{\partial x^m} -2H^{(\nu_0)}_{\mu\nu}\right){\Bigg )}\ast G_{\nu_0}
\end{array}
\end{equation}
where
\begin{equation}\label{hmunuend}
\begin{array}{ll}
g^{(\nu_0)}_{00}H^{(\nu_0)}_{\mu\nu}\equiv g^{(\nu_0)}_{00} H_{\mu\nu}\left(g^{(\nu_0)}_{\alpha\beta},\frac{\partial g^{(\nu_0)}_{\alpha\beta}}{\partial x^{\gamma}} \right)=g^{(\nu_0)}_{00}{\Big (}g^{(\nu_0)\alpha\beta}g_{\delta\epsilon}\Gamma^{(\nu_0)\delta}_{\mu\beta}\Gamma^{(\nu_0)\epsilon}_{\nu\alpha}
\\
\\
+\frac{1}{2}\left(\frac{\partial g_{\mu\nu}}{\partial x^{\alpha}}\Gamma^{\alpha}+g_{\nu\rho}\Gamma^{\rho}_{\alpha\beta}g^{\alpha\eta}g^{\beta\sigma}\frac{\partial g_{\eta\sigma}}{\partial x^{\mu}} +g_{\mu\rho}\Gamma^{\rho}_{\alpha\beta}g^{\alpha\eta}g^{\beta\sigma}\frac{\partial g_{\eta\sigma}}{\partial x^{\nu}}\right){\Big )}\in O\left(\frac{1}{r^2}\right).
\end{array}
\end{equation}
This implies that in the iteration scheme the convoluted terms of $h_{\mu\nu}$ are in $L^2$ (although not in $L^1$) for simension $n=3$. This holds als  for the  other two terms in (\ref{harm2end}). Analgous statements hold for the first order spatial derivatives of $h_{\nu\nu}$. Since the Gaussian is $L^1$- integrable and  Lipschitz continuous functions convoluted with first order spatial derivatives of the Gaussian are $L^1$-integrable, the iteration scheme can also be applied to general admissable data.

In summary the preceding argument gives counterexamples to the conjecture in \cite{P} where precise statements and characterisations can be found in  \cite{HE,W} and in \cite{N1,N2}. The method allows to construct counterexamples which are generically $L^p$-stable and have different stability properties compared to the singularities constructed in \cite{Cr1,Cr2}. We note furthermore that the existence ressult is not covered by the result stated in \cite{HKM}.  

\footnotetext[1]{\texttt{{kampen@mathalgorithm.de}, {kampen@wias-berlin.de}}.}

\end{document}